\documentclass{amsart}
\usepackage{amsmath,amssymb}
\usepackage{hyperref}
\usepackage{cleveref}
\usepackage{verbatim}
\usepackage{graphicx}
\usepackage{comment}
\newtheorem{theorem}{\bf Theorem}[section]
\newtheorem{condition}[theorem]{\bf Condition}%[section]
\newtheorem{remark}[theorem]{\bf Remark}%[section]
\newtheorem{lemma}[theorem]{\bf Lemma}%[section]

\crefname{condition}{condition}{conditions}
\Crefname{condition}{Condition}{Conditions}

\newcommand{\real}{\mathbb{R}}
\newcommand{\znat}{\mathbb{Z}}

\newcommand{\Mb}[1]{\left[{#1}\right]}

\newcommand{\captionfont}{}
\DeclareMathOperator{\Ein}{Ein}
\DeclareMathOperator{\E}{E}

\renewcommand{\setminus}{-}

\begin{document}
\title{Active Thermal Cloaking and Mimicking}

\author[M. Cassier, T. DeGiovanni, S. Guenneau, F. Guevara Vasquez]{%%%% Author details
Maxence Cassier$^{1}$, Trent DeGiovanni$^{2}$, S\'ebastien Guenneau$^{3}$, and Fernando Guevara Vasquez$^{2}$}

%%%%%%%%% Insert author address here
\address{$^{1}$Aix Marseille Univ, CNRS, Centrale Marseille, Institut Fresnel, Marseille, France\\
$^{2}$University of Utah, Mathematics Department, Salt Lake City UT 84112, USA\\
%$^{3}$Department of Mathematics, Imperial College London
$^{3}$UMI 2004 Abraham de Moivre-CNRS, Imperial College London, London SW7 2AZ, UK
}

%%%% Subject entries to be placed here %%%%

%%%% Keyword entries to be placed here %%%%

%%%% Insert corresponding author and its email address}

%%%% Abstract text to be placed here %%%%%%%%%%%%
\begin{abstract}
We present an active cloaking method for the parabolic heat (and mass or light diffusion) equation that can hide both objects and sources. By active we mean that it relies on designing monopole and dipole heat source distributions on the boundary of the region to be cloaked. The same technique can be used to make a source or an object look like a different one to an observer outside the cloaked region, from the perspective of thermal measurements. Our results assume a homogeneous isotropic bulk medium and require knowledge of the source to cloak or mimic, but are in most cases independent of the object to cloak. % 105 words
\end{abstract}
\maketitle

%%%%%%%%%%%%%%%%%%%%%%%%%%%

%%%%%%%%%% Insert the texts which can accomdate on firstpage in the tag "fmtext" %%%%%
%\begin{fmtext}
%%%%%%%%%%%%%%%%%%%%%%%%%%%%%%%%%%%%%%%%%%%%%%%%%%%%%%%%%%%%%%%%%%%%%%%%
\section{Introduction}\label{sec:intro}
Certain solutions to the heat equation can be reproduced inside or outside a closed surface by a source distribution on the surface determined by Green's identities. In particular, given a solution to the heat (or mass or light diffusion) equation in a homogeneous medium  and with no sources inside of a domain, it is possible to reproduce it inside the domain with a distribution of sources on the surface of the domain, while also giving a zero solution outside.  We call this the {\em interior reproduction problem}, see \cref{fig:cartoon1}(a).  Similarly, the {\em exterior reproduction problem} is to reproduce a solution to the heat equation in a homogeneous medium with no sources outside of a domain, while keeping a zero solution inside the domain (see \cref{fig:cartoon1}(b)). As we shall see, a growth condition  for the heat equation solution is needed to guarantee that the exterior reproduction problem can be solved. 
%\end{fmtext}
%%%%%%%%%%%%%%% End of first page %%%%%%%%%%%%%%%%%%%%%

This growth condition plays the same role as a radiation boundary condition for the Helmholtz equation (\cref{sec:irheq}). By combining solutions to interior/exterior reproduction problems we can achieve cloaking or mimicking for the heat equation in the following scenarios.

%%%%%%%%%%%%%%%%%%%%%%%%%%%%%%
\begin{figure}
    \centering
    \begin{tabular}{cc} 
     \includegraphics[width=30mm]{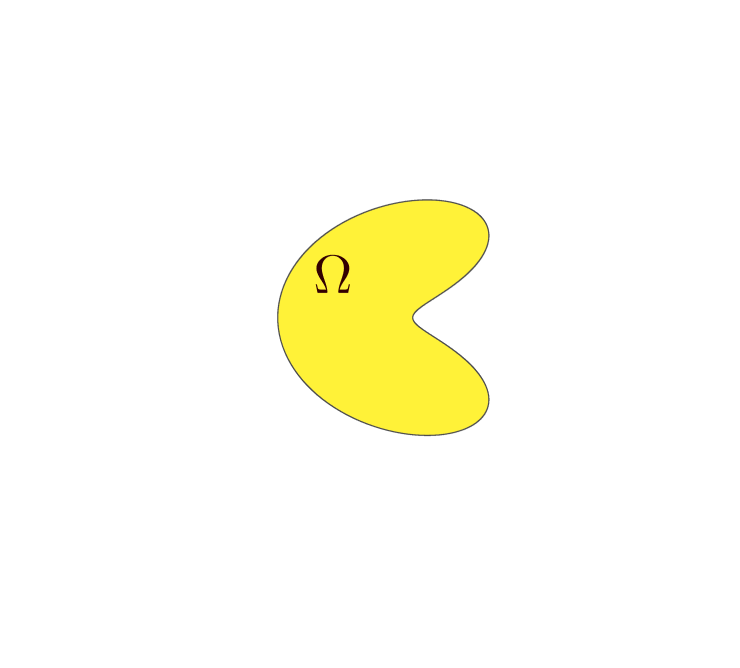} &
     \includegraphics[width=30mm]{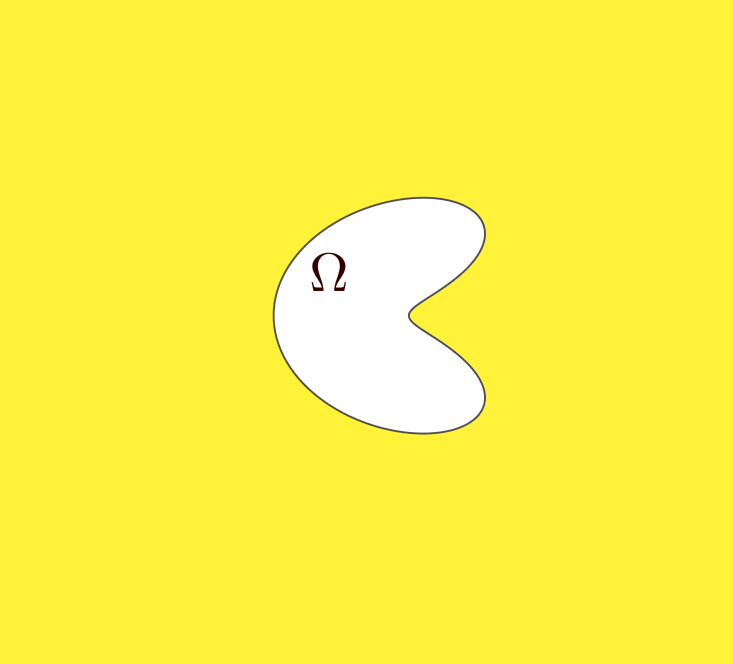} \\
    {\captionfont (a)} &
    {\captionfont (b)}\\
    \end{tabular}
    \caption{In (a) we illustrate the ``interior reproduction problem'' which consists of reproducing a solution to the heat equation in the interior of a bounded region $\Omega$ (in yellow), while enforcing a zero solution outside of $\overline{\Omega}$ (in white) by placing heat sources on the boundary $\partial\Omega$. In (b) we illustrate the ``exterior reproduction problem'' in a similar way.}
    \label{fig:cartoon1}
\end{figure}

{\bf Interior cloaking of a source}: Given a localized heat source, find an active surface or cloak surrounding the source so that the source cannot be detected by thermal measurements outside the cloak (\cref{sec:cloak:src}).

{\bf Interior cloaking of an object}:  Given a passive object (e.g. an inclusion), find an active surface or cloak surrounding the object so that the temperature distribution outside the cloak is indistinguishable from having a region of homogeneous medium instead of the cloak and the object (\cref{sec:cloak:obj}).

{\bf Source mimicking problem:} Given a localized source, find an active surface or cloak surrounding it so that the source appears as a different source for an observer outside the cloak (\cref{sec:mimic:src}).

{\bf Object mimicking problem:} Given a passive object  inclusion, find an active surface or cloak surrounding it so that the object appears as a different object for an observer outside the cloak (\cref{sec:mimic:obj}).

Active sources for the ``interior cloaking of an object'' problem have already been proposed and demonstrated experimentally for the steady state heat equation in \cite{Nguyen:2015:ATC}. The idea is to use Peltier elements to dissimulate a hole in a conductive plate under a steady state temperature distribution. A Peltier element is a thermoelectric heat pump that moves heat by putting an electric current across a metal/metal junction (see e.g. \cite{DiSalvo:1999:TCP}). Our approach deals with time varying solutions to the heat equation. It could be implemented with a setup similar to \cite{Nguyen:2015:ATC} by using Peltier devices that are either used to transport heat within the plate or act as heat sources/sinks on the plate by transporting heat between the environment and the plate, as illustrated in \cref{fig:peltier}. An alternative route using a single active dipole source placed inside the object to cloak in a constant gradient steady state regime is proposed in \cite{Xu:2019:DAT}. 

%%%%%%%%%%%%%%%%%%%%%%%%%%%%%%
\begin{figure}
    \centering
    \includegraphics[width=0.4\textwidth]{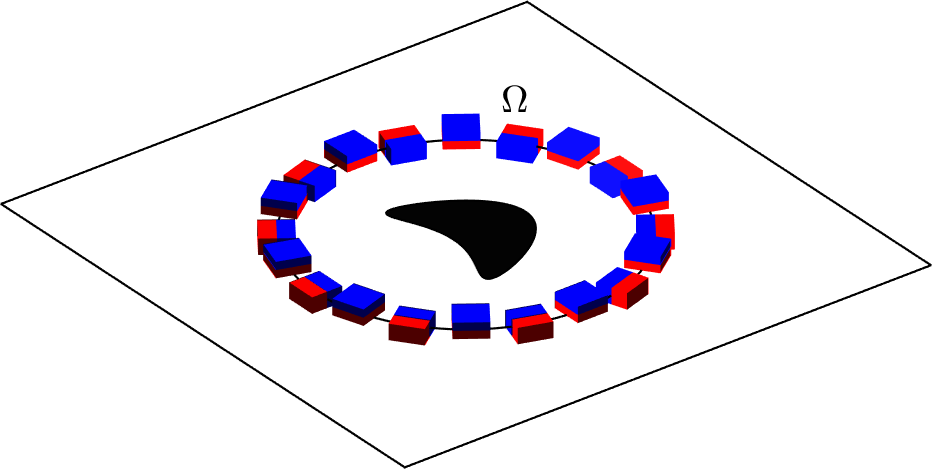}
    \caption{An illustrative example of an arrangement of Peltier elements that could be used to cloak objects (e.g. a kite) inside a two-dimensional region $\Omega$, illustrated here by a disk within a heat conducting plate. Each Peltier device is represented by two adjacent red and blue boxes, where the heat flux it can create is oriented in the direction normal to their interface. We have represented both Peltier elements that transport heat within the plate (across $\partial \Omega$) and also between the exterior and the plate.}
    \label{fig:peltier}
\end{figure}

Although all the results are presented in the context of the heat equation, they also apply to the diffusion equation, which can be used to model e.g. diffusion of a species in a porous medium. In this case the active surface would consist of pumps that can transport the species either across the medium or between the medium and the environment. For the diffusion equation case, we can either hide or imitate a source or an inclusion with different diffusivity properties.

Controlling the heat flux may find applications in enhancing the efficiency of thermal devices in solar thermal collectors, protecting electronic circuits from heating, or the design of thermal analogs of electronic transistors, rectifiers, and diodes \cite{Peralta:2019:CHM}. Moreover, all our results could be easily adapted to control of mass diffusion with potential applications ranging from biology with the delay of the drug release for therapeutic applications \cite{Guenneau:2013:FSL,Puvirajesinghe:2017:FSL} to civil engineering with the control of corrosion of steel in reinforced concrete structures\cite{Zeng:2013:CID}. We further note that in many media, such as clouds, fog, milk, frosted glass, or media containing many randomly distributed scatterers, light is not described by the macroscopic Maxwell equations, but rather by the Fick's diffusion equation as photons of visible light perform a random walk. Cloaking for diffusive light was experimentally achieved in \cite{Schittny:2014:DLSM} using the transformed Fick's equation \cite{Guenneau:2013:FSL} and this suggests potential applications of our work in control of diffusive light as well.

Because we are using source distributions, we can achieve cloaking of inclusions and sources, regardless of how complicated they are and on arbitrarily large time intervals. One drawback of our method is that the source distribution completely surrounds the object or source that we want to cloak or mimic. Another drawback is that we assume perfect knowledge of the fields to reproduce on a surface, however we expect this can be relaxed, as we demonstrate numerically in \cref{sec:irheq:num}. Apart from the active cloaking strategies for the steady state heat equation in \cite{Nguyen:2015:ATC,Xu:2019:DAT}, there are passive cloaking methods for the heat equation that use carefully crafted materials to hide objects \cite{Guenneau:2012:TTC,Ma:2013:TTC,Craster:2018:CMH,Schittny:2013:ETT,Guenneau:2013:FSL}. Such materials, whose effective conductivity mimics that in the heat equation after a suitable change of variables has been made, are quite bulky. In \cite{Schittny:2013:ETT}, some proof of concept of passive thermal cloaking was achieved with a metamaterial cloak consisting of 10 concentric layers mixing copper and polydimethylsiloxane in a copper plate. However, it has been numerically shown using homogenization in \cite{Petiteau:2014:SEETC} that one would require over 10,000 concentric layers with an isotropic homogeneous conductivity to accurately mimic the required anisotropic heterogeneous conductivity within a thermal cloak in order to achieve some markedly improved cloaking performance in comparison with \cite{Schittny:2013:ETT}.

The idea of using active sources based on the Green identities to cloak objects was first proposed for waves by Miller~\cite{Miller:2005:OPC}. The way of finding the sources for active cloaking is similar to that in active sound control for e.g. noise suppression \cite{Ffowcs:1984:RLA,Elliot:1991:ACO}. One problem with this approach is that the sources completely surround the cloak. However only a few sources are needed to cloak as was shown in \cite{Guevara:2009:BEO,Guevara:2009:AEC,Guevara:2011:ECA, Norris:2012:SAF, Guevara:2012:MAT, Guevara:2013:TEA} for the Laplace and Helmholtz equations. The approach can be extended to elastic waves \cite{Norris:2014:AEC} and flexural (plate) waves \cite{O'Neill:2015:ACO,McPhedran:2016:ACO}. The active cloaking approach can be applied to the steady state diffusion equation with the role of sources and sinks over a thick coating played by chemical reactions\cite{Avanzini:2020:CC}. Active cloaking has been demonstrated experimentally for the Laplace equation \cite{MA:2013:EOA,Ma:2016:OAC}, electromagnetics \cite{Selvanayagam:2013:EDO} and the steady state heat equation \cite{Nguyen:2015:ATC,Xu:2019:DAT}. Finally the illusion/mimicking problem was first proposed using metamaterials via a transformation optics approach \cite{Lai:2009:IOT} and then with active exterior sources \cite{Zheng:2010:EOC}. 

We start in \cref{sec:irheq} by recalling results on representing solutions to the heat equation by surface integrals. This section includes a growth condition on heat equation solutions that is sufficient to ensure that the exterior reproduction problem is solvable and numerical experiments illustrating both the interior and exterior reproduction problems in two-dimensions. We also underline properties related to the 
maximum principle of the heat equation that we use on one hand to point out some stability of the two reproduction
problems and on the other hand to interpret the numerical error in our simulations.
Then in \cref{sec:cloak} we explain how to cloak a source (\cref{sec:cloak:src}) or an object (\cref{sec:cloak:obj}). The mimicking problem is presented in \cref{sec:mimic} for both mimicking objects (\cref{sec:mimic:obj}) and sources (\cref{sec:mimic:src}). The numerical method we use to illustrate our approach in two-dimensions is explained in \cref{sec:numerics}. Finally, our results are summarized in \cref{sec:summary}.

%%%%%%%%%%%%%%%%%%%%%%%%%%%%%%%%%%%%%%%%%%%%%%%%%%%%%%%%%%%%%%%%%%%%%%%%
\section{Integral representation of heat equation solutions}
\label{sec:irheq}
We start by recalling results on boundary integral representation of solutions to the heat equation. Concretely we show in sections~\cref{sec:irheq:int} and \cref{sec:irheq:ext} how to use a distribution of monopole and dipole heat sources on a closed surface, an ``active surface'', to reproduce a large class of solutions to the heat equation inside and outside the surface. We include a two-dimensional numerical study (\cref{sec:irheq:num}) of how the reproduction error is affected by the time step, the boundary discretization and errors in the boundary density we use to represent the fields.

The temperature $u(x,t)$ of a homogeneous isotropic body satisfies the heat equation
\begin{equation}\label{eqn:heat}
\rho c u_t = \kappa  \Delta u +\widetilde{h},~\text{for }~t > 0,
\end{equation}
where $u$ is in Kelvin, $x$ is the position in meters, $t$ is the time in seconds, $\kappa$ is the thermal conductivity (W m$^{-1}$K$^{-1}$), $c$ is the specific heat (JK$^{-1}$kg$^{-1}$), $\rho$ is the mass density (kg m$^{-3}$) and $\widetilde{h}(x,t)$  is a source term (W m$^{-3}$).  To simplify the exposition we consider instead
\begin{equation}\label{eqn:heat_bis}
u_t = k  \Delta u +h,~\text{for }~t > 0,
\end{equation}
where $k = \kappa/\rho c$ is the thermal diffusivity (m$^2$ s$^{-1}$) and $h = \widetilde{h}/\rho c$ is the source term (K s$^{-1}$). In dimension $d$ the Green function or heat kernel for \eqref{eqn:heat_bis} is 
\begin{equation}
K(x,t) = 
\begin{cases}
    (4\pi kt)^{-d/2}\exp[-|x|^2/4kt], &\text{for}~t>0 \mbox{ and } x\in \real^d, \\
        0, &\text{otherwise},
    \end{cases}
    \label{eqn:fundamental_solution}
\end{equation}
where $| \cdot |$ is the Euclidean norm in $\real^d$. We point out that outside the origin $(x,t)=(0,0)$, $K(x,t)$ is a smooth function even on the line $t=0$. This can be shown directly or by using the hypoellipticity  property of the heat operator. Namely, as the heat kernel  solves the homogeneous heat equation in the distributional sense on any open set that does not contain the origin, by hypoellipticity (see \cite{Lions:1973:PDE} theorem 1.1 page 192) $K(x,t)$ can be extended as a $C^{\infty}$ function on $(\real^d\times \real)\setminus\{(0,0)\}$. Thus $K(x,t)$ is a smooth solution of the heat equation on any open subset of this set.

We consider a bounded non empty open set $\Omega$ with Lipschitz boundary $\partial \Omega$ and possibly multiple connected components in $\real^d$, $d \geq 2$. The interior reproduction problem (\cref{sec:irheq:int}) is to reproduce a solution of the homogeneous heat equation in the space-time cylinder $\Omega \times (0,\infty)$ by placing appropriate source densities on its boundary $\partial \Omega \times [0,\infty)$ and possibly on $\Omega \times \{0\}$ (initial condition). The sources are chosen such that the fields vanish in $(\real^d - \overline{\Omega}) \times (0,\infty)$ (where $\overline{\Omega}=\Omega \cup \partial \Omega$ denotes the closure of $\Omega$). If the initial condition is harmonic, it is sufficient to have sources on $\partial\Omega \times [0,\infty)$ only. For the exterior reproduction problem (\cref{sec:irheq:ext}) we seek to reproduce  a solution of the homogeneous heat equation in $(\real^d -\overline{ \Omega}) \times (0,\infty)$ using sources on $\partial\Omega \times [0,\infty)$ and possibly on $(\real^d -\overline{ \Omega}) \times \{0\}$. The fields are required to vanish in $\Omega \times (0,\infty)$.  

%%%%%%%%%%%%%%%%%%%%%%%%%%%%%%%%%%%%%%%%%%%%%%%%%%%%%%%%%%%%%%%%%%%%%%%%
\subsection{Reproducing fields in the interior of a bounded region}
\label{sec:irheq:int}
The goal here is to reproduce solutions $u$ to \eqref{eqn:heat_bis} in $\Omega$ by controlling sources on $\partial \Omega$ while leaving the exterior unperturbed. More precisely, for some temperature field $u(x,t)$, we wish to generate $u_\Omega(x,t)$ such that for $t> 0$:
\begin{equation}
    u_\Omega(x,t) = \begin{cases}
        u(x,t), & x \in \Omega, \\
        0, & x \notin \overline{\Omega}.
    \end{cases}
    \label{eq:uomega}
\end{equation}
Notice that $u_\Omega(x,t)$ is not defined for $x \in \partial\Omega$, as usual in boundary integral equations. 
For initial condition $u(x,0) = f(x)$, $x \in \Omega$ and a source term $h(x,t)$ supported in $\real^2-\Omega$, this can be achieved via the Green identities (see e.g.  \cite{Pina:1984:ATH,Costabel:1990:BIO,Friedman:1964:PDE})
\begin{equation}
\begin{split}
u_\Omega(x,t)= 
\int_0^t ds\int_{\partial \Omega} dS(y) [\frac{\partial u}{\partial n}(y,s)K(x-y,t-s)-u(y,s)\frac{\partial K}{\partial n}(x-y,t-s)]\\ 
+\int_{\Omega}f(y)K(x-y,t)dy,~~t>0,
\end{split}\label{eqn:brep}
\end{equation}
here $n(y)$ is the outward pointing unit length normal vector at $y \in \partial \Omega$. The Green identities guarantee that $u_\Omega(x,t)=0$ for $x \notin \overline{\Omega}$, as desired. The first term in the integral \eqref{eqn:brep} is the single layer potential (a collection of monopole heat sources) and the second is the double layer potential (a collection of dipole heat sources).  For zero initial conditions, the solution $u$ is completely represented within $\Omega$ by the single and double layer potentials, with  densities given by the field to be reproduced and its normal derivative on  $\partial\Omega$.
We point out that the representation formula \eqref{eqn:brep} holds for instance if  $u \in C^2(\overline{\Omega}\times [0,+\infty))$.  A less restrictive condition, for zero initial condition, is to assume some 
Sobolev regularity for $u$, see e.g. \cite[theorem 2.20]{Costabel:1990:BIO}. In this less regular setting, the first integral in $\eqref{eqn:brep}$ becomes a duality pairing between boundary Sobolev spaces.

\begin{remark}[Causality and instantaneous control]
\label{rem:causality}
The boundary integral representation \eqref{eqn:brep} is {\em causal} in the sense that to reproduce $u_{\Omega}(x,t)$, $t>0$, we only need information about $u$ in the past, i.e. for times before the present time $t$. Moreover, the time convolution \eqref{eqn:brep} suggests that at the present time $t$, we only require control of heat sources localized in time to the present time $t$.  Indeed, the integral over $\partial\Omega$ in \eqref{eqn:brep} is a collection of monopole and dipole sources localized in time to $s$ and that depends only on knowing $u$ and $\partial u / \partial n$ at time $s$. Moreover the contribution of past $s$, i.e. with $s<t$, amounts to the memory effect of the bulk.
Thus for experimental purposes, the boundary integral representation could be approximated by  e.g. Peltier devices.
\end{remark}

A numerical example is given in \cref{fig:int_greens}. Here the field $u$ is generated by a point source at $x=(0.25,0.25)$ and $t=0$. For the heat equation we took $k=0.3$ and the domain is $\Omega=B(x_0,r_0)$, the open ball with center $x_0 = (0.5,0.5)$ and radius $r_0=0.25$. It should be noted that in all of the numerical examples the thermal diffusivity $k$ is chosen for the convenience of computations, and may be different depending on the numerical experiment. We computed the fields on the unit square $[0,1]^2$ with a $200\times 200$ uniform grid at time $t=0.2$ s. The integral \eqref{eqn:brep} is approximated using the midpoint rule in time with 200 equal length subintervals of $[0,t]$  and the trapezoidal rule on $\partial \Omega$ with 128 uniformly spaced points. A more detailed explanation of the numerical method appears in \cref{sec:numerics}. The accuracy of our numerical method with respect to discretization changes and noise is evaluated in \cref{sec:irheq:num}. \Cref{fig:int_greens}(c) shows a plot of the $\log_{10}$ error between the computed field and the desired one. It can be observed that the accuracy of the numerical method improves as we move away from $\partial\Omega$.

%%%%%%%%%%%%%%%%%%%%%%%%%%%%%%
\begin{figure}
    \centering
    \begin{tabular}{ccc}
    \includegraphics[width=.3\textwidth]{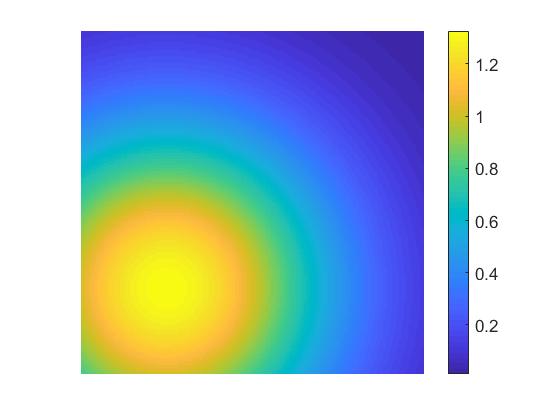} &
    \includegraphics[width=.3\textwidth]{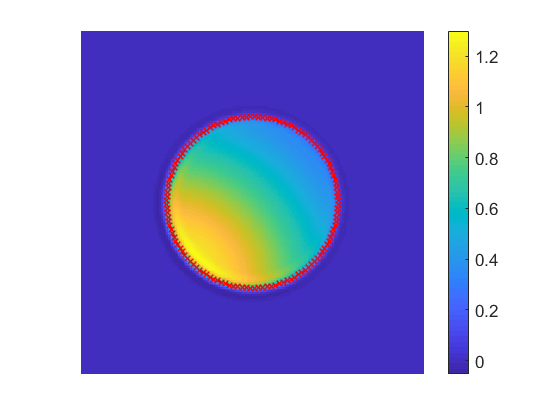} &
    \includegraphics[width=.3\textwidth]{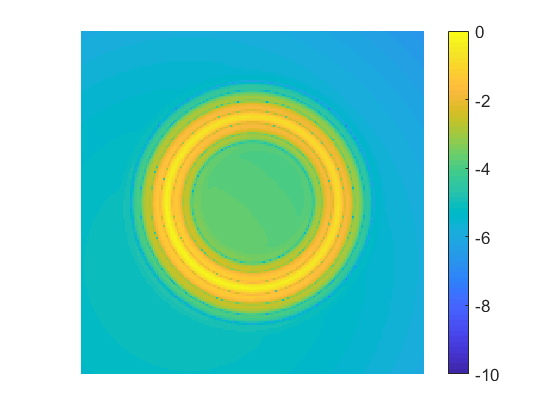}\\
    {\captionfont (a) Original field} 
    & {\captionfont (b) Reconstructed field} 
    & {\captionfont (c) $\log_{10}$ plot of errors}
    \end{tabular}
    \caption{Numerical example of the interior field reproduction problem for a point source located at $(1/4,1/4)$ and $t=0$~s. A snapshot of the original field at time $t=0.2$~s appears in (a). The field is reconstructed inside a disk of radius $1/4$ centered at $(1/2,1/2)$ by using only heat sources on the corresponding circle. In (c) we show the $\log_{10}$ of the reconstruction error (the absolute value of the difference between the exact $u_\Omega$ and its numerical approximation). We generated a plot similar to (c) by taking the maximum over the time interval $[0.2,0.3]$ at each grid point (the plot being very similar to (c), we include the code to generate it as supplementary material). This indicates that the maximum error is attained near $\partial\Omega$ in space (and in time at $t=0.2$ on the time interval $[0.2,0.3]$), conforming to the maximum principle (applied with initial time $t=0.2$) for the interior problem (\cref{rem.maximunprincpboundeddomain}) and the exterior one (\cref{rem.maximunprincpunboundeddomain}). 
    We can use the maximum principle on the computed fields because the numerical method we use generates solutions to the heat equation (see \cref{sec:numerics} for more details).
    }
    \label{fig:int_greens}
\end{figure}

We point out that the term in \eqref{eqn:brep} involving the initial condition is very different from the other terms because it is an integral over $\Omega$ rather than just the boundary $\partial \Omega$. 
If the initial condition $f(x)$ is non-zero but harmonic, this integral over $\Omega$ can also be expressed as an integral over $\partial\Omega$ by using Green identities as follows (see \cite{Pina:1984:ATH})
\begin{equation}
\int_{\Omega}f(y)K(x-y,t)dy = \int_{\partial \Omega}dS(y)[f(y)\frac{\partial \phi}{\partial n}(x-y,t)-\phi(x-y,t)\frac{\partial f}{\partial n}(y)],
\label{eqn:initial}
\end{equation}
where $\phi(x,t)$ is given in two dimension by  \cite{Pina:1984:ATH}:
    $$
        \phi(x,t)= \frac{-1}{4\pi}\Ein\bigg(\frac{r^2}{4kt} \bigg),
    $$
where $\Ein(z) = \E(z) + \ln z + \gamma$, $\gamma$ is the Euler constant and 
   \[
   \E(z) = \int_z^{+\infty} \frac{e^{-\zeta}}{\zeta} d \zeta
   \]
   is the exponential integral (see e.g. \cite[eq. 6.2.2 and 6.2.4]{NIST:DLMF}). The expression for $\phi$ in three dimensions is given in \cite{Pina:1984:ATH}. From the experimental perspective, it is not clear whether the kernels in the boundary integrals \eqref{eqn:initial} can be achieved with heat monopole and dipole sources. So we assume from now on, that the initial condition is harmonic but does not need to be reproduced using sources on $\partial\Omega$. Under this assumption, we can simply subtract the initial condition to obtain the heat equation with zero initial condition, which is the case that we focus on.
\begin{remark}\label{rem.maximunprincpboundeddomain}
As formula \eqref{eqn:brep} uses the boundary data and the initial condition  to express $u_{\Omega}$, this reconstruction of $u$ satisfies some stability due the maximum principle applied to the heat equation (a property that does not hold e.g. for the wave equation).
Indeed, for any $T>0$, the maximum principle \cite{Brezis:1964:PDE,Evans:2010:PDE,Kress:2014:LIE} states that a continuous function on $\overline{\Omega}\times [0,T]$ that solves the homogeneous heat equation \eqref{eqn:heat_bis} on $\Omega\times (0,T)$ (in the distributional sense) reaches its minimum and maximum  either at $t=0$ or at any time $t\in [0,T]$ on the boundary $\partial\Omega$. For instance, the solution of the initial-Dirichlet boundary value problem (with no sources) described in the  chapter 7 page 171-172  of \cite{Kress:2014:LIE}  satisfies the above conditions. In particular, this solution is continuous up to the boundary, i.e. on $\overline{\Omega}\times [0,T]$. To obtain such continuity, the continuous initial data and the continuous Dirichlet data have to match on $\partial \Omega$ at $t=0$, see e.g. \cite{Kress:2014:LIE}. 

To understand the stability, we take two solutions $u_{j}\in C^0(\overline{\Omega}\times [0,T])$, $j=1,2$,  that satisfy the homogeneous heat equation  \eqref{eqn:heat_bis}  in $\Omega\times (0,T)$ for $T>0$ in the distributional sense.
In this setting  solutions in the distributional sense are also smooth solutions of the homogeneous heat equation \eqref{eqn:heat_bis} (since by hypoellipticity,  $u_j \in C^{\infty}(\Omega\times (0,T))$ for $j=1,2$, see e.g. \cite{Lions:1973:PDE} theorem 1.1 page 192).
 Thus, one can apply  the maximum principle (see \cite{Brezis:1964:PDE}, theorem 10.6 page 334) to obtain that
\begin{equation}\label{eq.max1boundeddomain1}
\max_{\overline{\Omega}\times [0,T]}|u_2(x,t)-u_1(x,t)|=\max_{(\Omega \times \{0 \})\cup (\partial \Omega \times [0,T])}|u_2(x,t)-u_1(x,t)|.
\end{equation}
Moreover if the initial conditions are harmonic, using the maximum principle for the Laplace equation gives 
\begin{equation}\label{eq.max1boundeddomain2}
\max_{\overline{\Omega}\times [0,T]}|u_2(x,t)-u_1(x,t)|=\max_{\partial \Omega \times [0,T]}|u_2(x,t)-u_1(x,t)|.
\end{equation}
Thus, in the space of solutions of the homogeneous heat equation  (with the regularity described above), an error committed on the initial condition $u(x,0)$ or the  boundary Dirichlet data of a solution $u$ (i.e. the dipole distribution) on $\partial \Omega \times [0,T]$ controls the error (in the supremum norm) in the reconstruction of $u_{\Omega}$ in  $\Omega\times (0,T]$.

Finally, we point out an important property that constrains the behavior of $u_2-u_1$ if the maximum  in \eqref{eq.max1boundeddomain1} is attained at a point not located at the boundary or at the initial time.
Under the additional assumption that $\Omega$ is connected, one shows based on the mean value property of the heat operator and a connexity argument (as in \cite{Evans:2010:PDE}, section 2.2.3, theorem 4 page 54-55), that if there exists a point $(x_0,t_0)\in (0,T]\times \Omega$ such that $|u_2-u_1|$ reaches its maximum at $(x_0,t_0)$ in \eqref{eq.max1boundeddomain1}, then $u_2-u_1$ has to be constant in $\overline{\Omega}\times [0,t_0]$. Indeed, for formula \eqref{eq.max1boundeddomain2}, this property holds also for $t_0=0$  since the initial condition is harmonic and the Laplace operator satisfies a mean value property.
\end{remark}

%%%%%%%%%%%%%%%%%%%%%%%%%%%%%%%%%%%%%%%%%%%%%%%%%%%%%%%%%%%%%%%%%%%%%%%%
\subsection{Reproducing fields exterior to a bounded region}
\label{sec:irheq:ext}
For the exterior reproduction problem we seek to reproduce solutions to \eqref{eqn:heat_bis} outside of $\overline{\Omega}$ by controlling sources on $\partial \Omega$, while leaving the interior unperturbed. That is, for some temperature field $v(x,t)$ solving the heat equation we wish to generate $v_\Omega(x,t)$, such that
\begin{equation}
 v_\Omega(x,t) = 
    \begin{cases}
       0, &x \in \Omega, \\
        v(x,t), & x \notin \overline{\Omega},
    \end{cases}
    \label{eq:vomega}
\end{equation}
for $t > 0$. Notice that we have used the subscript $\Omega$ differently in \eqref{eq:vomega} than in \eqref{eq:uomega}. We adhere to the convention that $u_\Omega$ always refers to the interior reproduction problem of heat equation solution $u$ and $v_\Omega$ refers to the exterior reproduction of a field $v$. The problem of reproducing $v_\Omega$ can be solved when the source term associated with $v(x,t)$ is supported in  $\overline{\Omega}$ for all time.  We only consider the case where $v(x,0) = 0$. Non-zero initial conditions are left for future studies.

Without further assumptions the exterior reproduction problem may not have a unique solution as can be illustrated by the one-dimensional non-uniqueness example by Tychonoff \cite{Tychonoff:1935:TDP}. Uniqueness for the Dirichlet problem can be guaranteed via the maximum principle  \cite[chapter 2, section 3, theorem 7]{Evans:2010:PDE} (see also \cref{rem.maximunprincpunboundeddomain}) or via Sobolev regularity estimates in space and time \cite{Arnold:1987:BIE,Noon:1988:SLH}. For the transmission problem, growth restrictions in the Laplace domain are used to prove uniqueness in \cite{Qiu:2019:TDB}. Here we want to establish a boundary representation formula for exterior solutions to the heat equation, which uses both Dirichlet and Neumann data. Such exterior representation formula has already been mentioned in \cite{Noon:1988:SLH,Qiu:2019:TDB,Dohr:2019:DPS}, but without giving an explicit growth condition on the heat equation solution that guarantees its validity. We give a growth condition for the heat equation, analogous to the Sommerfeld radiation condition for the Helhmholtz equation \cite{Colton:2013:IAE} (a comparison between these two conditions is in \cref{rem:growth}).

To prove a boundary potential formula for the exterior reproduction problem, we follow the same steps as in \cite{Colton:2013:IAE} for the Helmholtz equation. Namely we use the interior reproduction problem (\cref{sec:irheq:int}) on the complement of $\overline{\Omega}$ truncated to a ball of sufficiently large radius $r$ (see \cref{fig:proof}), and then give a sufficient condition guaranteeing that the contribution from the sources at $|x| =r$ vanishes as $r \to \infty$, allowing heat equation solutions outside of $\overline{\Omega}$ to be reproduced by only controlling sources at $\partial\Omega$. The sufficient condition that we impose on the growth of $v(x,t)$ is close to the growth condition for the exterior Dirichlet problem uniqueness, see \cref{rem:growth}.
\begin{condition}[Growth condition]
 \label{cond:rad1}
 A differentiable function $v(x,t)$, where $x \in \real^d$,  $d\geq 2$, is said to satisfy the ``growth condition'' if there exists an $r_0 > 0$, such that if $r>r_0$,
\begin{equation}
 \left| \frac{\partial v}{\partial n}(r\xi, t)+\left(\frac{2r}{4kt}\right)v(r\xi, t) \right|  \leq  C r^m e^{a \, r^b}, \;\;\; \; \forall t>0, \; \xi \in S^d(0,1),
 \label{eqn:rad_con1}
\end{equation}
where $m$ is an integer, $C>0$, $a\geq 0$ are constants, the exponent $b$ satisfies $0\leq b<2$  and $S^d(0,1)$ is the sphere of radius $1$ centered at the origin in $d$ dimensions.
\end{condition}

\begin{remark}
\label{rem:growth}
The bound in \cref{cond:rad1}, is not as restrictive as the Sommerfeld radiation condition for the Helmholtz equation because of the Gaussian spatial decay of the heat kernel at fixed positive time. Indeed the Green function for the Helmholtz equation decays in space faster than $r^{-m}$ as $r\to \infty$, where $m>0$ depends on the dimension and $r = |x|$. On the other hand, the heat kernel decays faster than $r^{m} e^{-a \,r^b}$ as $r \to \infty$, with $a>0$, $0\leq b<2$ and $m \in \znat$. This is how we motivate the bound in \cref{cond:rad1}. However, we conjecture that it may be possible to improve the bound to allow $b=2$ (Gaussian decay over a polynomial), which would bring it on the same par as the growth condition guaranteeing uniqueness for the Dirichlet problem outside of a bounded domain (\cref{rem.maximunprincpunboundeddomain}). We point that \cref{cond:rad1}  is satisfied by a large class of solutions of the heat equations \eqref{eqn:heat} in $(\real^d\setminus \overline{\Omega}) \times (0,+\infty)$ that includes  in particular the heat kernel and any of its spatial derivatives (see \cref{lem:heatkernel}). 
\end{remark}

%%%%%%%%%%%%%%%%%%%%%%%%%%%%%%
\begin{figure}
\centering
\includegraphics[width=30mm]{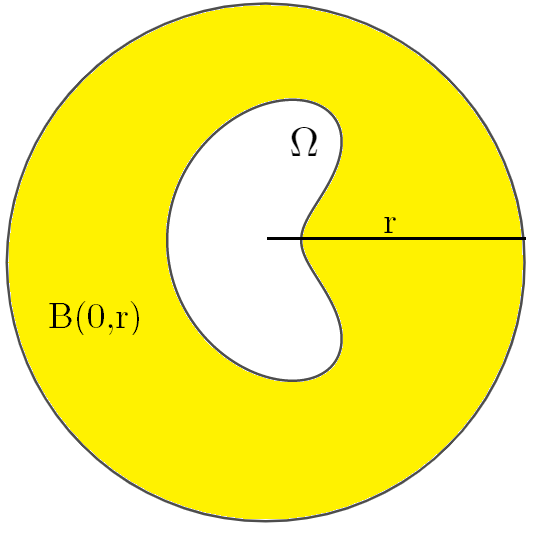}
\caption{To show that the exterior reproduction problem has a solution outside of a bounded region $\Omega$ we use the interior reproduction problem in the yellow region. \Cref{thm:rad_cond} shows that for fields satisfying \cref{cond:rad1}, the contribution of the sources on the sphere $S^d(0,r)$ of radius $r$ vanishes as $r\to\infty$ for $d\geq 2$.}
\label{fig:proof}
\end{figure}

%%%%%%%%%%%%%%%%%%%%%%%%%%%%%%
\begin{theorem}
\label{thm:rad_cond}
 Let  $v\in C^2\big((\real^d\setminus \Omega)\times [0,+\infty)\big)$ be a solution to the  heat equation  \eqref{eqn:heat} in $(\real^d\setminus \overline{\Omega})\times (0,+\infty) $ for $d\geq 2$,
 with zero initial condition and a source term spatially supported in the compact set $\overline{\Omega}$. Furthermore, assume that $v$ satisfies the growth \cref{cond:rad1}, then $v$ can be reproduced for $t>0$ in the exterior of $\overline{\Omega}$ by the boundary representation formula
\begin{equation}
v_{\Omega}(x,t)= -\int_0^t ds \int_{\partial \Omega} dS(y) [ \frac{\partial v}{\partial n}(y,s)K(x-y,t-s)-v(y,s)\frac{\partial K}{\partial n}(x-y,t-s)],
\label{eqn:outer_brep}
\end{equation}
where $v_{\Omega}$ is as in \eqref{eq:vomega}.
\end{theorem}
%%%%%%%%%%%%%%%%%%%%%%%%%%%%%%

\begin{proof}
Let $(x,t)\in (\real^d\setminus \overline{\Omega}) \times (0,+\infty)$ be fixed. Without loss of generality we assume that $0 \in \Omega$ and take $r$ to be large enough such that
 $\overline{\Omega} \subset B(0,r)$ and $x\in B(0,r)\setminus\overline{\Omega}$. Since $v\in C^2\big((\real^d\setminus \Omega)\times [0,+\infty)\big)$ satisfies the heat equation with a zero source term outside of $\overline{\Omega}$, we can use \eqref{eqn:brep} with zero initial condition to reproduce $v(x,t)$ on the bounded open set $B(0,r)\setminus\overline{\Omega}$, giving
\begin{equation}
\label{eqn:irep:out}
v_{B(0,r)\setminus \overline{\Omega}}(x,t)=  I_1(x,t;r) + I_2(x,t)
\end{equation}
where 
\begin{equation}
 \begin{aligned}
 I_1(x,t;r) &= \int_0^t ds \int_{S(0,r)} dS(y) \Mb{\frac{\partial v}{\partial n}(y,s)K(x-y,t-s)-v(y,s)\frac{\partial K}{\partial n}(x-y,t-s)},~\text{and}\\
 I_2(x,t) &= -\int_0^t ds \int_{\partial \Omega} dS(y) \Mb{ \frac{\partial v}{\partial n}(y,s)K(x-y,t-s)-v(y,s)\frac{\partial K}{\partial n}(x-y,t-s)}.
\end{aligned}
\end{equation}
The minus sign in $I_2(x,t)$ is because we defined $n$ as the outward pointing unit normal to $\Omega$. 
The goal is now to show that $I_1(x,t;r) \to 0$ as $r \to \infty$, leaving us with only $I_2(x,t)$ which gives the desired result \eqref{eqn:outer_brep}. We rewrite $I_1(x,t;r)$ using 
\[ 
 \frac{\partial K}{\partial n}(x,t) = K(x,t) \Big(\frac{-2x}{4kt} \cdot n\Big),
\]
and switching the convolutions in time to get
\begin{equation}\label{eq.Green}
 I_1(x,t;r) = \int_0^t ds \int_{S(0,r)} dS(y) \Mb{\frac{\partial v}{\partial n}(y,t-s)+\bigg(\frac{2y}{4k(t-s)} \cdot n\bigg)v(y,t-s) } K(x-y,s).
\end{equation}
We define $\xi = x/|x|$ ($x \neq 0$ since $x\notin \overline{\Omega}$). Thus, for $y \in S(0,r)$, we can bound the heat kernel by
\begin{equation}
\label{eq.bound}
 K(x-y,s) \leq K(x - r\xi,s),
\end{equation}
because $|x-y|\geq |x-r\xi|$ holds for $|y| = r$. Noticing that we also have $y \cdot n =r$ for $|y|=r$, we can use \cref{cond:rad1} to bound $I_1(x,t;r)$ for sufficiently large $r$. 
Thus, using \eqref{eq.Green}, the bound \eqref{eq.bound} and applying \cref{cond:rad1} leads to
\begin{equation}
|I_1(x,t;r)| \leq  C r^m e^{a \, r^b} A_d(r) \, \int_0^t ds\, \frac{1}{(4\pi ks)^{d/2}} e^{-|x-r\xi|^2/4ks},
\label{eq:bound1}
\end{equation}
where $A_d(r)$ is the surface of a sphere of radius $r$ in $d$ dimensions, which is given in terms of the Gamma function (see e.g. \cite[eq. 5.2.1]{NIST:DLMF}) by
\[
 A_d(r)=\frac{2\pi^{d/2}}{\Gamma(d/2)}r^{d-1}.
\]

Now using the change of variables $u = |x-r\xi|^2/4ks$ on the integral appearing in the right hand side of \eqref{eq:bound1} yields:
\begin{equation}
\label{eq:incomplete}
\begin{aligned}
\int_0^t ds[s^{-d/2} e^{-|x-r\xi|^2/4ks}]&=\int_{+\infty}^{|x-r\xi|^2/4kt}du\Mb{\frac{(4ku)^{d/2}}{|x-r\xi|^d} \left(\frac{-|x-r\xi|^2}{4ku^2}\right) e^{-u} } \\
&=\frac{(4k)^{d/2-1}}{|x-r\xi|^{d-2}}\int^{+\infty}_{|x-r\xi|^2/4kt}du[u^{d/2-1-1}e^{-u}] \\
& = \frac{(4k)^{d/2-1}}{|x-r\xi|^{d-2}}\Gamma\bigg(\frac{d}{2}-1,\frac{|x-r\xi|^2}{4kt}\bigg), 
\end{aligned}
\end{equation}
where the upper incomplete Gamma function $\Gamma(d/2-1,\cdot)$ is defined for all $y>0$ by:
\begin{equation}\label{eq:defincmpleteGamma}
\Gamma\Big(\frac{d}{2}-1,y\Big) =\int_{y}^{+\infty} du[u^{d/2-1-1}e^{-u}].
\end{equation}
In our case $y=|x-r\xi|^{d-2}/(4kt) \to + \infty$ as $r\to+\infty$. Thus, we need an equivalent of $\Gamma(d/2-1,y)$  as $y\to +\infty$. To this aim, we do an integration by parts on \eqref{eq:defincmpleteGamma} to get :
\begin{equation}\label{eq:integbypart}
\Gamma\Big(\frac{d}{2}-1,y\Big) =e^{-y} y^{d/2-2}+ (d/2-2) \int_{y}^{+\infty} du [u^{d/2-1-1} u^{-1} e^{-u}]. 
\end{equation}
Then, as $u\geq y$, one observes that
$$
\Big| \int_{y}^{+\infty} du [u^{d/2-1-1} u^{-1} e^{-u}] \Big| \leq \frac{1}{y} \Gamma\Big(\frac{d}{2}-1,y\Big)
$$
and thus concludes from \eqref{eq:integbypart} that:
\begin{equation}\label{eq:equivalent}
\Gamma\Big(\frac{d}{2}-1,y\Big)= e^{-y} y^{d/2-2}(1 + o(1)),~~\text{as}~y\to+\infty.
\end{equation}

Combining  \eqref{eq:bound1}, \eqref{eq:defincmpleteGamma}  and the equivalent of the incomplete Gamma function \eqref{eq:equivalent} for $y=|x-r\xi|^{d-2}/(4kt)$  gives that for $r$ large enough:
\begin{eqnarray}
|I_1(x,t;r)|& \leq& 2 C \, r^m e^{a r^b} A_d(r) \frac{(4k)^{d/2-1}}{|x-r\xi|^{d-2}} \Big(\frac{|x-r\xi|}{4kt} \Big)^{d/2-2}  e^{-|x-r\xi|^2/4kt} \nonumber \\
&\leq &\widetilde{C}_{x,t}\, r^{m+d/2-1} e^{-|x-r\xi|^2/4kt+ a r^b},
\label{eq:bound3}
\end{eqnarray}
where $\widetilde{C}_{x,t}$ is positive constant that depends only $x$ and $t$ that are here fixed.
To conclude, observe that the upper bound in \eqref{eq:bound3} goes to 0 as $r \to \infty$ since $b<2$. This statement holds for any $x$ outside of $\overline{\Omega}$ and any $t>0$, yielding the representation \eqref{eqn:outer_brep}. 
\end{proof}
\begin{remark}\label{regularity}We point out that the regularity assumption  $v\in C^2\big((\real^d\setminus \Omega)\times [0,+\infty))$ of the solution in \cref{thm:rad_cond} can be relaxed. Indeed, our proof still holds with weaker regularity assumptions but for a smooth bounded open set $\Omega$ (i.e. with a $C^{\infty}$ boundary $\partial \Omega$). For instance, our proof works under a  Sobolev local regularity, namely if the solution $v$ (in the sense of distributions) belongs to $ H^{2,1}(\mathcal{O}\times (0,T)) $ for any $T>0$ and any  open bounded set $\mathcal{O}\subset\real^d\setminus{\overline{\Omega}}$ (we refer to \cite{Costabel:1990:BIO,Lions:1972:PDE} for the definition of $H^{2,1}$). 
This assumption and the zero initial condition of $v$ allows us to apply the representation formula \eqref{eqn:irep:out}  in the proof for $\mathcal{O}=B(0,r)\setminus \overline{\Omega}$ by applying the theorem 2.20 of \cite{Costabel:1990:BIO}. In this setting the integrals in \eqref{eqn:irep:out}  have to be interpreted as duality pairings between Sobolev spaces of the boundary $\partial\Omega$ (see \cite{Costabel:1990:BIO} for more details).
By the trace theorem 2.1 page 9 in \cite{Lions:1972:PDE}, the assumed local Sobolev regularity ensures that $v$ and $\partial v/\partial n$ belong to  $L^2(\partial \Omega \times (0,t))$ and $L^2(\partial B_r\times (0,t))$ for $t>0$ and $r$ sufficiently large. Thus, the integrals $I_1(x,t;r)$ and $I_2(x,t;r)$ can be interpreted not only as a duality paring but as  integrals.
Furthermore, by interior regularity (and even hypoellipticity) of the differential operator in the heat equation (see  \cite{Lions:1973:PDE} theorem 1.1 page 192), one has 
$v \in C^{\infty}(\real^d\setminus \overline{\Omega})\times (0,+\infty)$. Thus, as $v$ is smooth on this set, the growth \cref{cond:rad1} is still well-defined.  The proof of \cref{thm:rad_cond} follows similarly and yields the representation formula \eqref{eqn:outer_brep} in this new setting.
\end{remark}

%Note that for fixed $r$, and $$g(r)=A_d(r)P_n(r)e^{ar}4k|x-r\xi|^{-2}$$ the limit of the bound is $$\lim_{t \rightarrow 0} g(r)t^2\frac{e^{-|x-r\xi|^2/4kt}}{(4\pi kt)^{d/2}} = \lim_{t \rightarrow 0} g(r)t^2K(x-r\xi,t)  =0 $$ which follows from $$\lim_{t \rightarrow 0} K(x-r\xi,t)= \delta(x-r\xi).$$ Hence, the bound is consistent with zero initial condition. 

%%%%%%%%%%%%%%%%%%%%%%%%%%%%%%
%\begin{remark}
 %\Cref{thm:rad_cond} also holds if 
%To prove \cref{thm:rad_cond} under this more restrictive condition, we proceed as in the original proof. If \eqref{eqn:bounded} holds then we need to modify the upper bound in \eqref{eq:bound1} by replacing the term in from the integral by $C$. Then by  \eqref{eq:incomplete} and $\Gamma(\alpha, z) \to 0$ as $z \to \infty$ for $\alpha >0$ we see that the new upper bound goes to zero. This shows that $I_1(x,t;r) \to 0$ as $r \to \infty$, and the desired result follows, with the additional benefit of removing the restriction on the dimension in \cref{thm:rad_cond}.
%\end{remark}

The heat kernel and its spatial derivatives (which all solve the heat equation) satisfy the growth \cref{cond:rad1} as we see next.
%%%%%%%%%%%%%%%%%%%%%%%%%%%%%%
\begin{lemma}
\label{lem:heatkernel}
In dimension $d \geq 1$, the heat kernel and all its spatial derivatives  satisfy the growth condition \eqref{eqn:rad_con1} for some $C>0$ and any $a\geq 0$, $b \in [0,2)$, $r_0>1$ and non-negative integer $m$.
\end{lemma} 
%%%%%%%%%%%%%%%%%%%%%%%%%%%%%%
\begin{proof}
We first introduce the notation $\partial_x^\alpha \psi$ for arbitrary spatial derivatives of a smooth function $(x,t) \mapsto \psi(x,t)$ on $\real^d \times (0,\infty)$:
\[
 \partial_{x}^\alpha \psi(x,t)=\partial_{x_1}^{\alpha_1}\,  \partial_{x_2}^{\alpha_2} \ldots \, \partial_{x_d}^{\alpha_d} \psi(x_1,x_2,\ldots,x_d,t), \ \mbox{ where } x=(x_1,x_2,\ldots,x_d).
\]
Here $\alpha = (\alpha_1,\ldots,\alpha_d)$ is a multi-index, where $\alpha_i$ is the order of differentiation in $x_i$. 

Since the heat kernel $K$ is in $C^{\infty}(\real^d\times (0,+\infty))$, an induction on the degree of differentiation reveals that for $(x,t)\in \Omega\times (0,+\infty)$, $K$ and any of its spatial derivatives have the form,
\begin{equation}\label{eq.deriv}
\partial_x^\alpha K(x,t) = 
P(x_1/t, x_2/t, \ldots,x_d/t,1/t) \, K(x,t), 
\end{equation} 
where $P$ is a multivariate polynomial. In the particular case $\alpha=(0,0,\ldots,0)$, we have $\partial_\alpha K=K$ and thus $P=1$. %with the following property. If $|\alpha|=0$, $\partial_\alpha K=K$ and $P=1$ and if $|\alpha|\neq 0$, the partial degree of $P$ is  $\alpha_i$ in $x_i/t$ and $\alpha_t= \samll \displaystyle (\max_{i=1,\ldots,d} \alpha_i)-1 \geq 0$ in $1/t$.}  %Further, the normal derivative, $\nabla(\partial_\alpha K(x,t))\cdot n$ has the same form but the order of the polynomial, considered in both $x$ and $1/t$, is $|\alpha|+1.$
Using the triangle inequality on the expression \eqref{eq.deriv} of $\partial_x^\alpha K$ leads to:
\begin{equation}\label{eq.tringineq}
|\partial_x^\alpha K(x,t)|\leq Q(|x|/t, |x|/t, \ldots,|x|/t,1/t) \, |K(x,t)| \quad  \mbox{ for } (x,t)\in \real^d \times (0,\infty) ,
\end{equation}
where $Q$ is a polynomial which has the same monomial terms as $P$, but whose coefficients are given by the modulus of the coefficients of $P$.

Let $r_0\geq 1$ and $|x|=r>r_0$. We set $u=|x|^2/t$ in the right hand of side of \eqref{eq.tringineq}. As  $|x|/t\leq |x|^2/t=u$ (since $1\leq |x|\leq |x|^2$), $1/t=u/|x|^2\leq u/r_0^2$ (since $|x|>r_0$) and the coefficients of $Q$ are positive, it follows from \eqref{eq.tringineq} and the expression \eqref{eqn:fundamental_solution} of $K$ that:
\begin{equation}\label{eq:bounderiv1}
|\partial_x^\alpha K(x,t)|\leq Q(u, u, \ldots,u,u/r_0^2) \,(4\, k\pi)^{-d/2} (u/r_0^2)^{d/2} e^{-u/(4k)}  \quad \mbox{ for  $|x|>r_0$ and $t$>0}.
\end{equation}
As $Q$ is a polynomial, due to the exponential term, the right hand side of \eqref{eq:bounderiv1} is clearly bounded for $u>0$, thus there is a constant $C_{1,\alpha}>0$ (depending only on $\alpha$, $d$ and $r_0$) such that:
\begin{equation}\label{eq:bounderiv2}
|\partial_x^\alpha K(x,t)|\leq C_{1,\alpha} \mbox{ for  $|x|>r_0$ and $t$>0} .
\end{equation}
Now, as the bound \eqref{eq:bounderiv1} holds for any $\alpha$, one immediately deduces that there is a $C_{2,\alpha}$ such that
\begin{equation}\label{eq:bounnormalderiv}
| \nabla(\partial_x^{\alpha} K(r\xi,t)) \cdot n |\leq |\nabla(\partial_x^{\alpha} K(r\xi,t))| \leq C_{2,\alpha} \ \mbox{ for $x=r\xi$, $\xi \in S^d(0,1)$,  $r>r_0$ and $t>0$.}
\end{equation}
Thus, by \eqref{eq:bounnormalderiv}, \eqref{eq:bounderiv1} and the bound $r/t\leq r^2/t=u$ (as $r> r_0 \geq  1$), there is a $C_{3,\alpha}>0$ such that:
\begin{equation}
    \begin{aligned}
        \Big|\nabla(\partial_\alpha K(r\xi,t)) \cdot n +\frac{r}{2kt} \partial_\alpha K(r\xi,t)\Big| &\leq  \Big|\nabla(\partial_\alpha K(r\xi,t)) \cdot n\Big| +\frac{u}{2k}\Big| \partial_\alpha K(r\xi,t)\Big| \\
        &\leq C_{2,\alpha}+\frac{(2k)^{-1}}{(4\pi k r_0^2)^{d/2}}\,  Q\Big(u, u, \ldots,u,\frac{u}{r_0^2}\Big) \, u^{d/2+1} e^{-u/(4k)}
        \\
        &\leq  C_{3,\alpha},
    \end{aligned}
    \label{eqn:hk_bigt}
\end{equation}
for any $r>r_0$, $\xi\in S^d(0,1)$ and $t>0$.
%%explanation for form of spatial derivatives
%\begin{comment}
%The form of this equation follows from two observations that are wholly mechanical. (1) If this is the first time taking a derivative with respect to $x_i$, then the derivative is the existing function multiplied by $x_i(2kt)^{-1}$. (2) If a derivative has already been taken with respect to $x_i$ there are two terms in the derivative. One where $x_i(2kt)^{-1}$ multiplies the original function, and one where the derivative of the polynomial is multiplied by $K(x,t)$. 
%\end{comment}
\end{proof}

\begin{remark}\label{rem.maximunprincpunboundeddomain}
As in \cref{rem.maximunprincpboundeddomain} for bounded $\Omega$, the representation formula \eqref{eqn:outer_brep} satisfies some stability due to the maximum principle. However, as $\real^d\setminus \overline{\Omega}$ is unbounded, it requires a bound that controls the growth of the functions when $|x|\to+ \infty $, namely, one assumes that there exist $A,a>0$ such that:
\begin{equation}\label{eq.bounduniqueness}
|u(x,t)|\leq A e^{a|x|^2}, \mbox{ for }(x,t)\in (\real^d\setminus \overline{ \Omega}) \times (0,T],
\end{equation}
for some finite $T>0$. This last condition allows Gaussian growth and is similar to \cref{cond:rad1}.

More precisely, let $v_{j}\in C^{0}\big((\real^d\setminus \Omega)\times [0,T]\big)$ for $j=1,2$ be two solutions  of the homogeneous heat equation in $(\real^d\setminus \overline{\Omega})\times (0,T)$ in the distributional sense that satisfy the growth condition \eqref{eq.bounduniqueness}. Again by hypoellipticity (see \cite{Lions:1973:PDE} theorem 1.1 page 192),  $u_j$ is indeed a smooth solution of the homogeneous heat equation on  
$(\real^d \setminus \overline{\Omega})\times (0,T)$ for $j=1,2$ since it is $C^{\infty}$ on this set.
Then by the maximum principle (see e.g. \cite[chapter 3, section 3, theorem 6]{Evans:2010:PDE}) one has:
\begin{equation}\label{eq.maxprincipunbounded}
\sup_{(\real^d\setminus \Omega) \times [0,T]}|v_2(x,t)-v_1(x,t)|=\sup_{\big((\real^d \setminus \Omega) \times \{0 \}\big)\cup \big(\partial \Omega \times [0,T]\big)}|v_2(x,t)-v_1(x,t)|.
\end{equation}
Note that the proof of \cite[chapter 3, section 3, theorem 6]{Evans:2010:PDE} is done on all of $\real^d$, but can be adapted to  $\real^d\setminus \overline{\Omega}$ with $\Omega$ a bounded Lipschitz domain. Furthermore if the initial condition is harmonic and decays to $0$ as $|x|\to +\infty$, using the maximum principle for the Laplace equation in unbounded domains,  one can simplify \eqref{eq.maxprincipunbounded} to include only surface terms in the right hand side
\begin{equation}\label{eq.maxprincipunbounded2}
\max_{(\real^d\setminus \Omega ) \times [0,T]}|v_2(x,t)-v_1(x,t)|=\max_{\partial \Omega \times [0,T]}|v_2(x,t)-v_1(x,t)|.
\end{equation}
Thus, in  the space of solutions of the homogeneous heat equation (that satisfy \eqref{eq.bounduniqueness} and the regularity described above),
both \eqref{eq.maxprincipunbounded} and \eqref{eq.maxprincipunbounded2} tell us that  an error committed on the initial condition or on the Dirichlet boundary data controls the reconstruction error of $v_{\Omega}$ on $(\real^d\setminus \overline{\Omega}) \times (0,T]$, in the supremum norm. Furthermore, uniqueness on $(\real^d\setminus\Omega)\times [0,T]$ for the heat equation exterior Dirichlet problem  follows from the maximum principle equality \eqref{eq.maxprincipunbounded}, provided the growth condition \eqref{eq.bounduniqueness} and regularity assumptions hold. Moreover if  \eqref{eq.bounduniqueness} is satisfied for any $t>0$, this uniqueness result extends to $(\real^d\setminus  \Omega) \times [0,+\infty)$, assuming the same regularity assumptions but with an infinite time.

Finally, as in  \cref{rem.maximunprincpboundeddomain},  
under the additional assumption that the open set  $\real^d\setminus \overline{\Omega}$ is connected, if a maximum is attained in \eqref{eq.maxprincipunbounded} at $(x_0,t_0) \in (\real^d\setminus \overline{\Omega}) \times (0,T]$  then there exists a real constant $C$ such that $v_2(x,t) - v_1(x,t)=C$  on $(\real^d\setminus \Omega) \times [0,t_0]$. Furthermore, one shows that if one considers the formula \eqref{eq.maxprincipunbounded2}, this last  property holds also for $t_0=0$ and the constant $C$ has to be zero (since in formula \eqref{eq.maxprincipunbounded2}, one assumes that the initial conditions decay to $0$ when $|x|\to +\infty$ which imposes that $C=0$).
\end{remark}

A numerical example of the exterior reproduction of a field can be seen in \cref{fig:ext_greens}. The details of the example are the same as those in \cref{fig:int_greens}, except the point source has been moved to $(0.5,0.55)$.

%%%%%%%%%%%%%%%%%%%%%%%%%%%%%%
\begin{figure}
    \centering
    \begin{tabular}{ccc}
        \includegraphics[width=.3\textwidth]{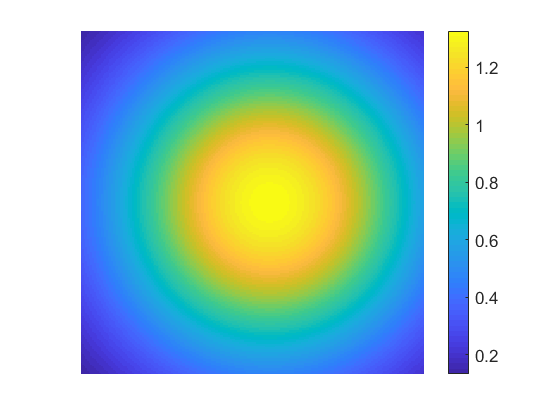} & \includegraphics[width=.3\textwidth]{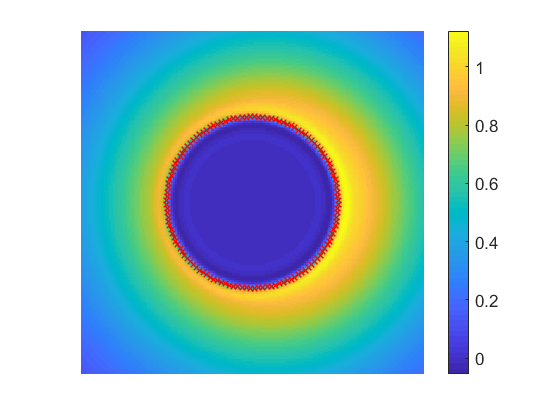} &
        \includegraphics[width=.3\textwidth]{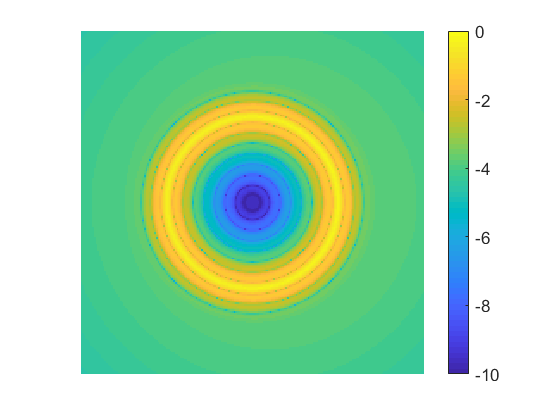} \\
        {\captionfont(a) Original field} & {\captionfont(b) Reproduced field} & {\captionfont(c) $\log_{10}$ plot of errors}
    \end{tabular}
    \caption{Numerical example of the exterior field reproduction problem. The details of this example are the same as \cref{fig:int_greens}, but with the point source moved to $(0.5,0.55)$.  We generated a plot similar to (c) by taking the maximum over the time interval $[0.2,0.3]$ for each grid point (the plot being very similar to (c), we include code to generate it as supplementary material).  This indicates that the maximum error is attained near $\partial\Omega$ in space (and in time at $t=0.2$), conforming to the maximum principle (\cref{rem.maximunprincpboundeddomain,rem.maximunprincpunboundeddomain}).
    %This is in accordance with the maximum principle for the interior problem (\cref{rem.maximunprincpboundeddomain}) and the exterior one (\cref{rem.maximunprincpunboundeddomain}) that the maximum error is reached near $\partial\Omega$ and at $t=0.2$ on the time interval $[0.2,0.3]$. 
    We can use here the maximum principle on the computed error because the numerical methods we use generates solutions to the heat equation (see \cref{sec:numerics}). }
    \label{fig:ext_greens}
\end{figure}

%%%%%%%%%%%%%%%%%%%%%%%%%%%%%%%%%%%%%%%%%%%%%%%%%%%%%%%%%%%%%%%%%%%%%%%%
\subsection{Numerical sensitivity study of field reproduction} 
\label{sec:irheq:num}

%%%%%%%%%%%%%%%%%%%%%%%%%%%%%%
\begin{figure}
\centering
\includegraphics[width=36mm]{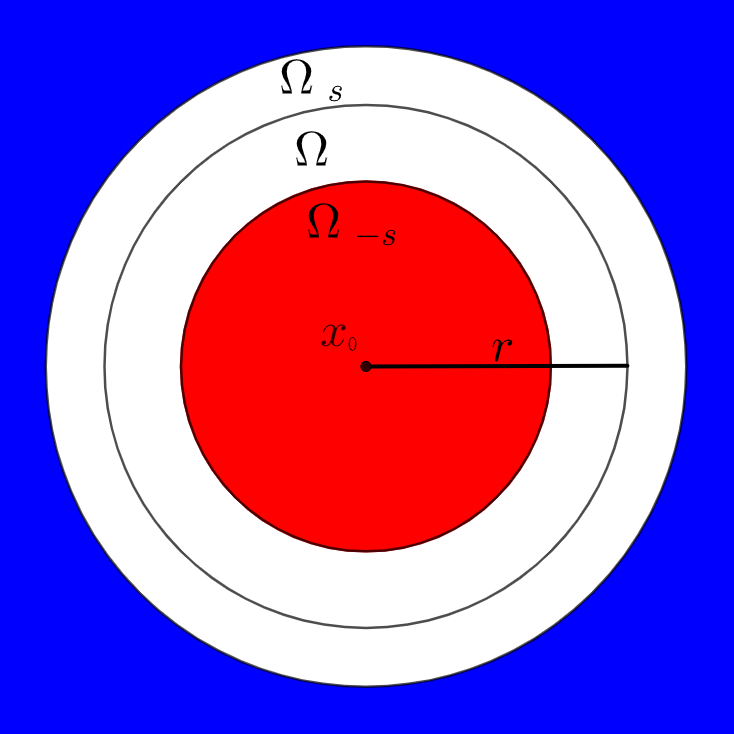}
\caption{We evaluate the field reproduction errors on regions that exclude a neighborhood of the boundary $\partial\Omega$. For the numerical experiments we took $\Omega = B(x_0,r) \subset [0,1]^2$. The interior reproduction error is evaluated on $\Omega_{-s} = B(x_0,(1-s)r)$ (in red) while the exterior one is evaluated on $[0,1]^2 - \Omega_s$ (in blue), where $\Omega_s = B(x_0,(1+s)r)$.} 
\label{fig:cart_err}
\end{figure}

We study the sensitivity of the numerical approximation of the boundary representation formulas  \eqref{eqn:brep} and \eqref{eqn:outer_brep} to the following factors: (a) the spatial discretization (number of points on $\partial \Omega$), (b) the temporal discretization (number of time steps) and (c) errors in the densities appearing in the boundary representation formulas. As can be seen in \cref{fig:int_greens}(c) and \cref{fig:ext_greens}(c), the reproduction error peaks close to the boundary, so we decided to exclude a neighborhood of $\partial\Omega$ from the error measures we present. The numerical approximation of the boundary reproduction formulas is explained in detail in \cref{sec:numerics}.
Here we keep the same domain $\Omega$ as in the examples of \cref{sec:irheq:int,sec:irheq:ext}. The boundary integral representations were used to approximate the field on a $100 \times 100$ uniform grid of $[0,1]^2$ and the thermal diffusivity was taken to be $k=0.2$.

The first case we consider is that of the interior reproduction problem, i.e., when the source distribution is supported in $\real^2-\Omega$. We expect using \eqref{eqn:brep} that $u_\Omega = u$ inside $\Omega$ and $u_\Omega = 0$ outside $\overline{\Omega}$. To evaluate the quality of the numerical approximation we make, we calculate the relative reproduction error on a slightly smaller domain $\Omega_{-s} = B(x_0,(1-s)r)$, 
\[
 \text{relerr}_-(u;t) = \frac{\| u(\cdot,t)-u_\Omega(\cdot,t) \|_{L^2(\Omega_{-s})}}{\| u(\cdot,t) \|_{L^2(\Omega_{-s})}},
\]
where we used the $L^2$ norm of a function over some set $R$, namely
\[
 \| f \|_{L^2(R)} = \Big( \int_R |f(x)|^2 dx \Big)^{1/2}.
\]
We also calculate the absolute error outside of a slightly larger domain $\Omega_s = B(x_0,(1+s)r)$, i.e. 
\[
 \text{err}_+(u;t) =\|u_\Omega(\cdot,t) \|_{L^2([0,1]^2-\Omega_s)}.
\]
The $L^2$ norms appearing in the error quantities that we consider are approximated using Riemann sums on the $100 \times 100$ grid of $[0,1]^2$. The domains of interest, $[0,1]^2-\Omega_{s}$ and $\Omega_{-s}$, are illustrated by the blue and red regions of \cref{fig:cart_err} respectively. In our numerical experiments we chose $s=.05$ to get a buffer annulus at $\pm 5\%$ of $r$. The field $u$ is generated by a delta source $\delta(x,t)$.

Similarly for the exterior reproduction problem, where we want to reproduce a field $v$ satisfying the heat equation with a source term supported in $\Omega$ and satisfying the radiation type boundary condition \eqref{eqn:rad_con1}, we calculate the absolute interior error
\[
 \text{err}_-(v;t) = \| v_\Omega(\cdot,t) \|_{L^2(\Omega_{-s})}
\]
and the relative exterior error
\[
 \text{relerr}_+(v;t) = \frac{\|v(\cdot,t) - v_\Omega(\cdot,t) \|_{L^2([0,1]^2 -\Omega_s)}}{\|v(\cdot,t)\|_{L^2([0,1]^2 -\Omega_s)}} .
\]
For the exterior reproduction studies, the field $u$ is produced by a delta source located at $x=(0.5,0.55)$ and $t=0$.

%%%%%%%%%%%%%%%%%%%%%%%%%%%%%%%%%%%%%%%%%%%%%%%%%%%%%%%%%%%%%%%%%%%%%%%%
\subsubsection{Sensitivity to spatial discretization} In \cref{fig:spatial_errors} we illustrate the changes in reproduction error for both the interior (\cref{fig:spatial_errors} first row) and exterior (\cref{fig:spatial_errors} second row) reproduction problems.  For both studies a uniform discretization of $\partial \Omega$ is used and 1000 uniform time steps. While increasing the number of points on $\partial \Omega$ decreases the error in all cases, the decrease from 50 to 100 points is modest. We think this is due to the temporal discretization error being dominant. 

%%%%%%%%%%%%%%%%%%%%%%%%%%%%%%
\begin{figure}
    \centering
    \begin{tabular}{ccc}
    & {\captionfont interior error} & {\captionfont exterior error}\\
    \raisebox{2em}{\rotatebox{90}{\captionfont Interior reproduction problem}} &
    \includegraphics[width=55mm]{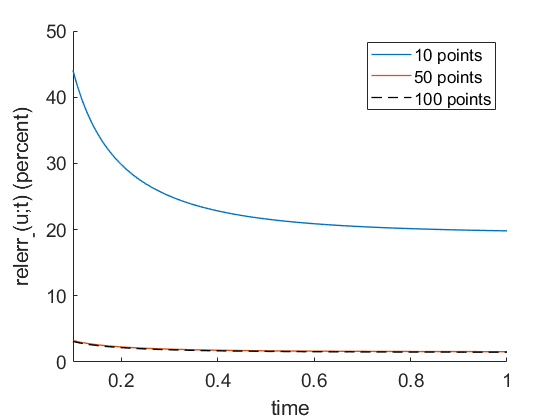} & \includegraphics[width=55mm]{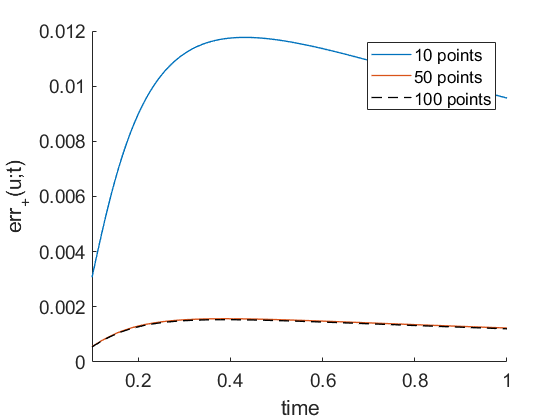} \\
    \raisebox{2em}{\rotatebox{90}{\captionfont Exterior reproduction problem}} &
    \includegraphics[width=55mm]{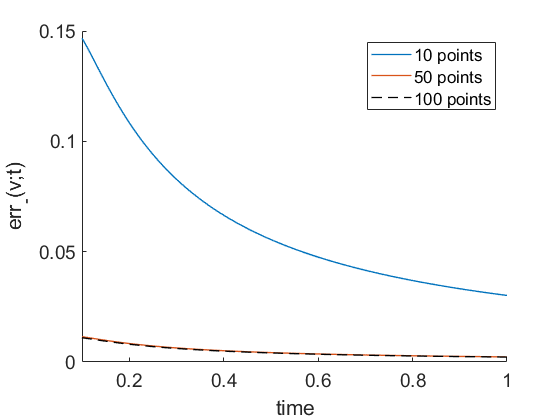} & \includegraphics[width=55mm]{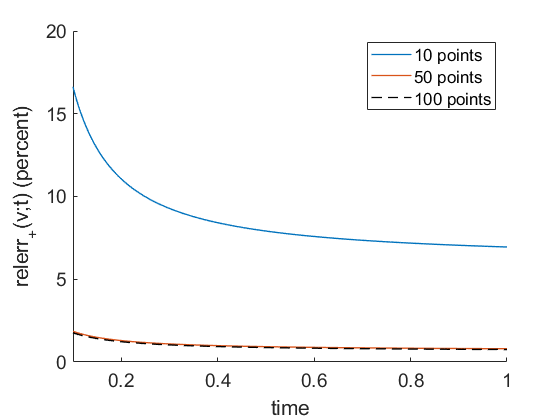} \\
    \end{tabular}
    \caption{Influence of the number of points used to discretize $\partial \Omega$ on the reproduction error for the interior reproduction problem of a point source located at $x_0=(0,0), t=0$ (top row) and for the exterior reproduction problem  of a point source located at $x_0=(0.5,0.55),t=0$ (bottom row), both with thermal diffusivity $k=0.2$. Here $\partial \Omega$ is the circle of radius $0.25$ centered at $(0.5,0.5)$. Since the errors for small times are large, we only show the errors for $t \geq 0.1$.}
    \label{fig:spatial_errors}
\end{figure}

%%%%%%%%%%%%%%%%%%%%%%%%%%%%%%%%%%%%%%%%%%%%%%%%%%%%%%%%%%%%%%%%%%%%%%%%
\subsubsection{Sensitivity to temporal discretization} We report in \cref{fig:temporal_errors} the change in reproduction error as we increase the number of time steps while keeping the number of uniformly spaced points used to discretize $\partial \Omega$ fixed and equal to $100$. This is done for both the interior (\cref{fig:temporal_errors} first row) and exterior (\cref{fig:temporal_errors} second row) reproduction problems. For a fixed time  the errors decrease with the number of time steps, as expected.

%%%%%%%%%%%%%%%%%%%%%%%%%%%%%%
\begin{figure}
    \centering
    \begin{tabular}{ccc}
         & {\captionfont interior error} & {\captionfont exterior error}\\
        \raisebox{2em}{\rotatebox{90}{\captionfont Interior reproduction problem}} & \includegraphics[width=55mm]{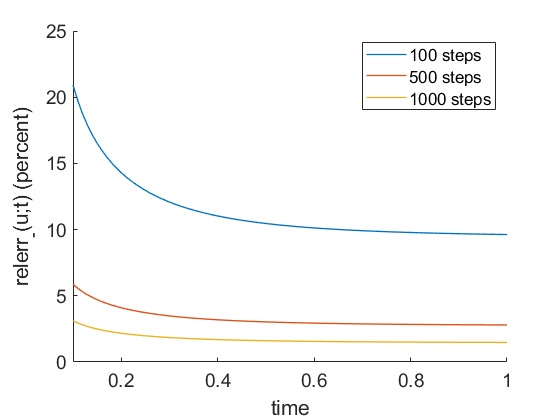} & \includegraphics[width=55mm]{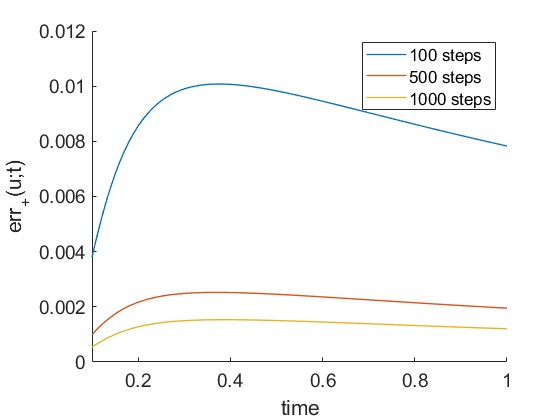} \\
        \raisebox{2em}{\rotatebox{90}{\captionfont Exterior reproduction problem}} &
         \includegraphics[width=55mm]{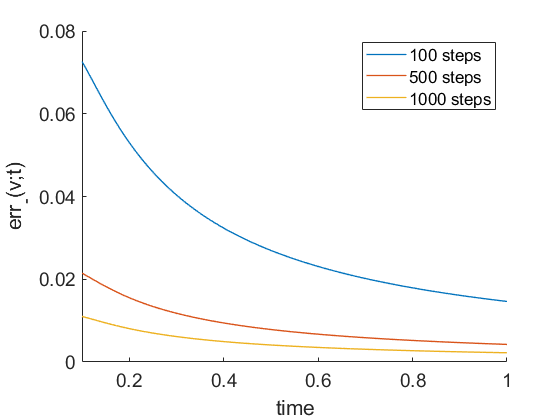} & \includegraphics[width=55mm]{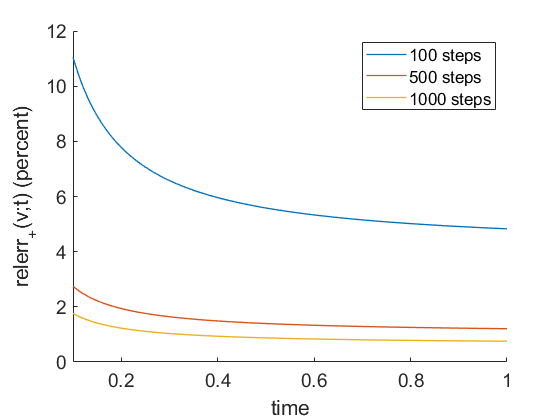} \\
    \end{tabular}
    \caption{Influence of the number of time steps on the reproduction error for the interior reproduction problem of a point source located at $x_0=(0,0), t=0$ (first row) and for the exterior reproduction problem of a point source located at $x_0=(0.5,0.55),t=0$ (second row), both with thermal diffusivity $k=0.2$. Here $\partial \Omega$ is the circle of radius $0.25$ centered at $(0.5,0.5)$. Since the errors for small times are large, we only show the errors for $t \geq 0.1$.}
    \label{fig:temporal_errors}
\end{figure}

%%%%%%%%%%%%%%%%%%%%%%%%%%%%%%%%%%%%%%%%%%%%%%%%%%%%%%%%%%%%%%%%%%%%%%%%
\subsubsection{Sensitivity to errors in the densities}
In practice it cannot be assumed that the field to be reproduced is perfectly known so we report in \cref{fig:errors} how the reproduction error is affected by errors in the monopole and dipole densities appearing in \eqref{eqn:brep} and \eqref{eqn:outer_brep} when discretized with 1000 time steps and 100 points on $\partial \Omega$. 

Say $\phi^{(n)} \in \real^{100}$ is a vector representing the values of either the monopole or dipole density in one of the boundary representation formulas at time $n\Delta t$, where $\Delta t$ is the time step. We perturb $\phi^{(n)}$ with a  vector $\delta \phi^{(n)} \in \real^{100}$ with independent identically distributed zero mean Gaussian entries with standard deviation being a fraction (3\%) of $\|\phi^{(n)}\|_2$.  For clarity we only show the error for a single realization of the perturbation $\delta \phi^{(n)}$. The errors we observe in  \cref{fig:errors} oscillate rapidly because we introduced a random perturbation at every single time step. As expected, the error we introduced in the densities increases the overall error at every single time step. We include code to generate the spatial distribution of the reconstruction errors as supplementary material (see {\tt README} file), we observe that the maximum error over a time window is attained  near $\partial\Omega$ as predicted by the maximum principle (\cref{rem.maximunprincpboundeddomain,rem.maximunprincpunboundeddomain}). Since the additive noise comes from perturbing the monopole and dipole densities, this perturbation is also a smooth solution of the heat equation (see \cref{sec:numerics}), which makes the maximum principle applicable.

%%%%%%%%%%%%%%%%%%%%%%%%%%%%%%
\begin{figure}
    \centering
    \begin{tabular}{ccc}
        & {\captionfont interior error} & {\captionfont exterior error}\\
       \raisebox{2em}{\rotatebox{90}{\captionfont Interior reproduction problem}} & \includegraphics[width=55mm]{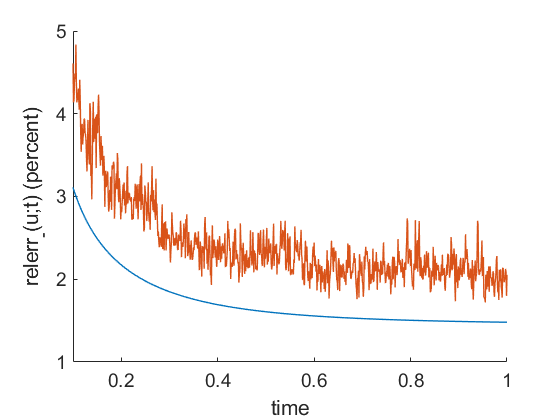} & \includegraphics[width=55mm]{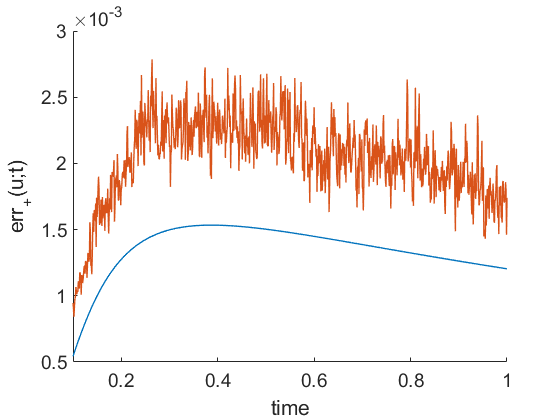}\\
       \raisebox{2em}{\rotatebox{90}{\captionfont Exterior reproduction problem}} &
        \includegraphics[width=55mm]{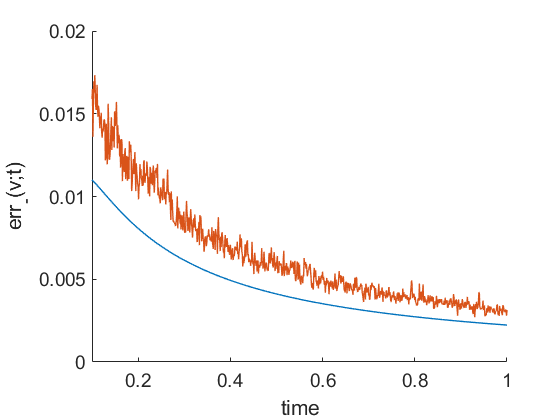} & \includegraphics[width=55mm]{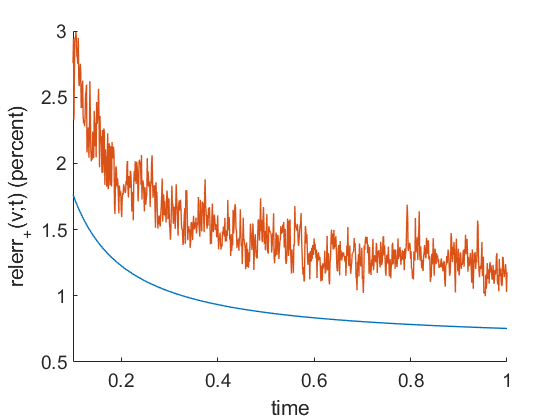} 
    \end{tabular}
    \caption{Influence of random perturbations added to the monopole and dipole densities on the reproduction error for the interior reproduction problem of a point source located at $x_0=(0,0), t=0$ (first row) and for the exterior reproduction problem of a point source located at $x_0=(0.5,0.55),t=0$ (second row), both with thermal diffusivity $k=0.2$. Here $\partial \Omega$ is the circle of radius $0.25$ centered at $(0.5,0.5)$. Since the errors for small times are large, we only show the errors for $t \geq 0.1$.  
    In orange: error with random perturbations. In blue: error obtained with the unperturbed densities.}
    \label{fig:errors}
\end{figure}

\begin{remark}
In the numerical results we present, the exterior errors, relative and absolute, are smaller than the interior errors. We also see a transient effect in the absolute exterior error. We believe this is similar to the temperature distribution near a point source, which also increases and then decreases with time.
Since the errors that we report in \cref{fig:spatial_errors,fig:temporal_errors,fig:errors} are based on the $L^2$ norm, the maximum principle considerations of \cref{rem.maximunprincpboundeddomain,rem.maximunprincpunboundeddomain} do not apply directly to these numerical experiments.
\end{remark}

%%%%%%%%%%%%%%%%%%%%%%%%%%%%%%%%%%%%%%%%%%%%%%%%%%%%%%%%%%%%%%%%%%%%%%%%
\section{Cloaking}
\label{sec:cloak}
The goal here is to use the results from  \cref{sec:irheq} to cloak sources or objects inside a cloaked region, by placing sources on the surface of the region. By cloaking we mean that it is hard to detect the object or source from only thermal measurements made outside the cloaked region. The boundary representation formulas of \cref{sec:irheq} give us the appropriate surface source distribution. We start in \cref{sec:cloak:src} with the interior cloaking of a source, directly applying the boundary representation formula in \cref{sec:irheq:ext}. The interior cloaking of an object is illustrated in \cref{sec:cloak:obj} by using the boundary representation formula in \cref{sec:irheq:int}. The boundary representation formulae impose restrictions on what can be cloaked and how. In either case the field must be known for all time and with no sources in the region where it is reproduced. For the interior cloaking of a source, the temperature field generated by this source must also satisfy \cref{cond:rad1}.

%%%%%%%%%%%%%%%%%%%%%%%%%%%%%%%%%%%%%%%%%%%%%%%%%%
\subsection{Cloaking a source in an unbounded domain}
\label{sec:cloak:src}
Given certain kinds of localized heat source distributions, we can find an active surface surrounding the source so that the source cannot be detected by an observer outside the surface. Let $v_i(x,t)$ be a free space solution to the heat equation \eqref{eqn:heat_bis} with zero initial condition and compactly supported source distribution $h(x,t)$. Let  $\Omega$ be an open bounded set (with Lipschitz boundary) that contains the support of the source $h(x,t)$ for $t>0$. In an analogy with wave problems, we call $v_i$ the ``incident field'' and we further assume that it satisfies the growth \cref{cond:rad1}. By \cref{thm:rad_cond}, we can find monopole and dipole densities on $\partial\Omega$ so that the boundary representation formula \eqref{eqn:outer_brep} gives $-v_i$ outside of $\overline{\Omega}$ and $0$ inside $\Omega$. We call this the cloaking field $v_c$ and it is given for $t>0$ by
\begin{equation}
   v_c(x,t) = \begin{cases} 0 &x \in \Omega\\
                    - v_i(x,t) &x \notin \overline{\Omega}.
    \end{cases}
    \label{eqn:ext_cloak}
\end{equation}
In this manner the total field $v_{\text{tot}} = v_i + v_c$ is zero outside of $\overline{\Omega}$ and equal to $v_i$ inside $\Omega$. Because the active surface $\partial \Omega$ perfectly cancels the effect of the source $h(x,t)$ for $x \notin \overline{\Omega}$, the source cannot be detected by an observer. \Cref{fig:ext_greens} shows a numerical example of $v_c$.

%%%%%%%%%%%%%%%%%%%%%%%%%%%%%%%%%%%%%%%%%%%%%%%%%%%%%%%%%%%%%%%%%%%%%%%%
\subsection{Cloaking passive objects in an unbounded domain}
\label{sec:cloak:obj}

One way to detect an object in free space using only thermal measurements would be to generate an incident or probing field $u_i(x,t)$ with a source distribution $h(x,t)$, i.e. a solution to the heat equation \eqref{eqn:heat_bis} in free space with zero initial condition and $h$ as its source term. 
In the presence of an object, the total field is given by $u_{\text{tot}} = u_i + u_s$, where $u_s$ is the field ``scattered'' by the object, borrowing terminology from the wave equation. The scattered field is produced by the interaction between the incident field and the object and depends on the properties of the object (boundary condition, heat conductivity, $\ldots$). We point out that $u_s(x,0)=0$ because $u_{\text{tot}}(x,0) = u_i(x,0)$. Having $u_s \neq 0$ reveals the presence of an object. In the following we assume that the object is ``passive'', meaning that the scattered field is linear in the incident field. In particular this means that $u_s = 0$ when $u_i = 0$. Examples of passive objects include objects with homogeneous linear boundary conditions (e.g. Dirichlet, Neumann or Robin) or objects  with a heat conductivity that is different from that of the surrounding medium (see e.g. \cite{Ammari:2005:DATI,Hohage:2005:NSH} for transmission problems for the heat equation).  We point out that the object is assumed to be open with Lipschitz boundary.

The results in \cref{sec:irheq} can be used to cloak a passive object $R$ by placing it inside a cloaking region $\Omega$ (i.e. a bounded open set $\Omega$ with smooth boundary such that $\overline{R}\subset \Omega$) and makes this whole region invisible from probing incident fields $u_i$ generated by a source $h$  spatially supported in $\real^d\setminus \Omega$. Indeed, by controlling monopoles and dipoles on $\partial\Omega$, the region $\Omega$ and the object within can be made indistinguishable from a patch of homogeneous medium, from the perspective of thermal measurements outside of $\overline{\Omega}$. In a similar manner to \cref{sec:cloak:src}, the idea is to use \eqref{eqn:brep} to cancel the incident field in  $\Omega$, while leaving the outside of $\overline{\Omega}$ unperturbed. The cloaking field, $u_c$, produced by this active surface $\partial \Omega$ is then for $t>0$:
\begin{equation}
  u_c(x,t) = \begin{cases} -u_i(x,t) &x \in \Omega,\\
                    0 &x \notin \overline{\Omega}.
    \end{cases}
    \label{eqn:int_cloak}
\end{equation}
In principle, this cloaking field can be used to perfectly cancel the incident field in $\Omega$ for all $t>0.$ Since the temperature of the ``modified incident field'': $u_i+u_c$ is zero in $\Omega$, the temperature field surrounding the object vanishes and no scattered field is produced.  In practice, the field $u_i + u_c$ in the vicinity of the object does not perfectly vanish, but we expect it to be sufficiently close to zero so that the scattered field $u_s$ is very small (because of linearity).

Our technique is illustrated with an object with homogeneous Dirichlet boundary conditions in \cref{fig:contour_cloaking}. Here the field $u_i$ is generated by a point source at $x=(0.9,0.3)$ and $t=0$. For the heat equation we took $k=0.2$ and  the cloaked region is $\Omega = B(x_0,r)$ with $x_0 = (0.5,0.5)$ and $r=1/3$. We computed the fields on the unit square $[0,1]^2$ with a $200\times 200$ uniform grid. The field $u_c$ is found by approximating the integral \eqref{eqn:brep} using the midpoint rule in time with 600 equal length subintervals of $[0,0.5]$  and the trapezoidal rule on $\partial \Omega$ with 128 uniformly spaced points. A more detailed explanation, including how the scattered fields are calculated, is in \cref{sec:numerics}. We represent in figure \cref{fig:contour_cloaking} the total fields respectively generated by the incident field $u_i$ (left column) and  by the "modified incident field": $u_i+u_c$ (right column). As can be seen in the right column, the temperature fields, outside of the cloaked region are indistinguishable from the incident field $u_i$. We do not include a detailed error plot for this configuration, as the error is similar to the one we encountered when studying the interior reproduction problem. 

\begin{remark}
\label{rmk:other}
 Here are two ways of dealing with {\em active objects}, i.e. that are not passive.  First if the object produces a non-zero scattered field $u_s$ when $u_i = 0$, the object acts as a source and $u_s$ can be cancelled using the technique in \cref{sec:cloak:src}. This presumes perfect knowledge of $u_s$ and $u_i$. Second, if the object does not produce a scattered field when immersed in some harmonic field $u_0$, we can use as a cloaking field $u_c(x,t) = -u_i(x,t) + u_0(x)$ for $x\in \Omega$ and $u_c(x,t) = 0$ for $x \notin \overline{\Omega}$, instead of \eqref{eqn:int_cloak}. An example of such an object would be one with a constant $c\neq 0$ Dirichlet boundary condition. By our assumption, the field $u_0(x) = c$ does not create any scattering, regardless of the shape of the object.
\end{remark}

%%%%%%%%%%%%%%%%%%%%%%%%%%%%%%
\begin{figure}
    \centering
    \begin{tabular}{cc}
    \includegraphics[width=50mm]{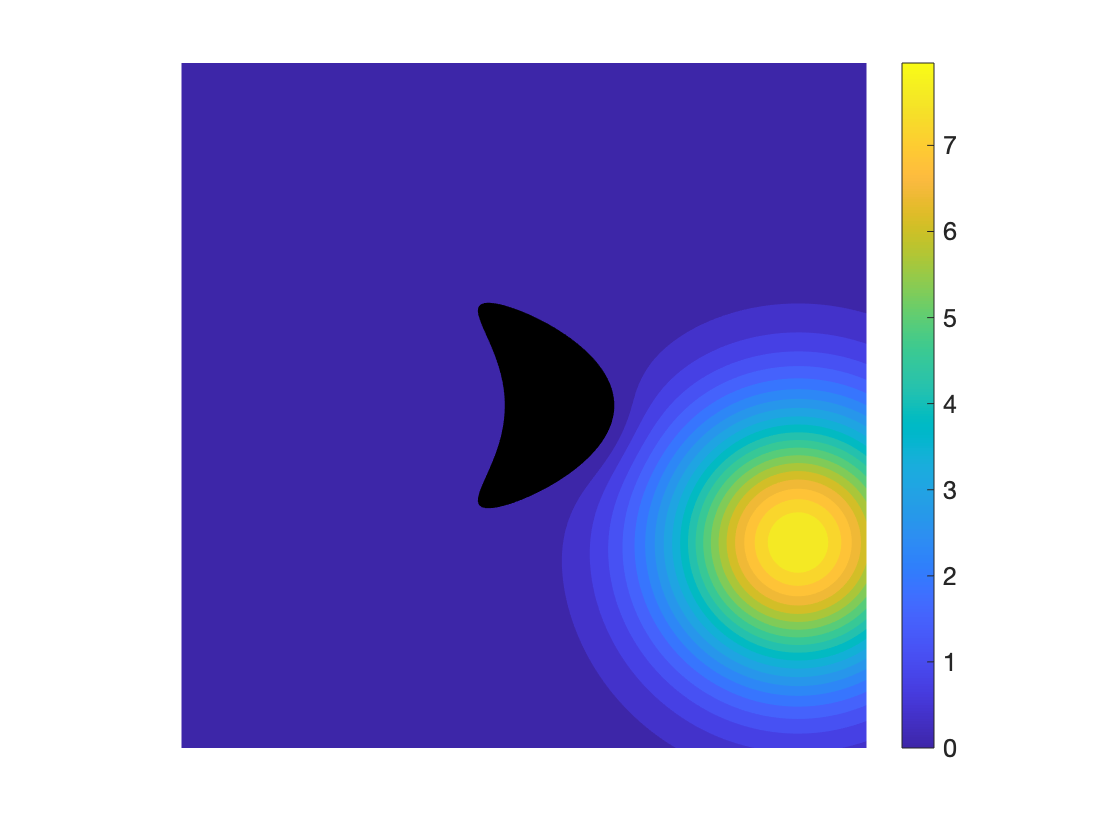} &
    \includegraphics[width=50mm]{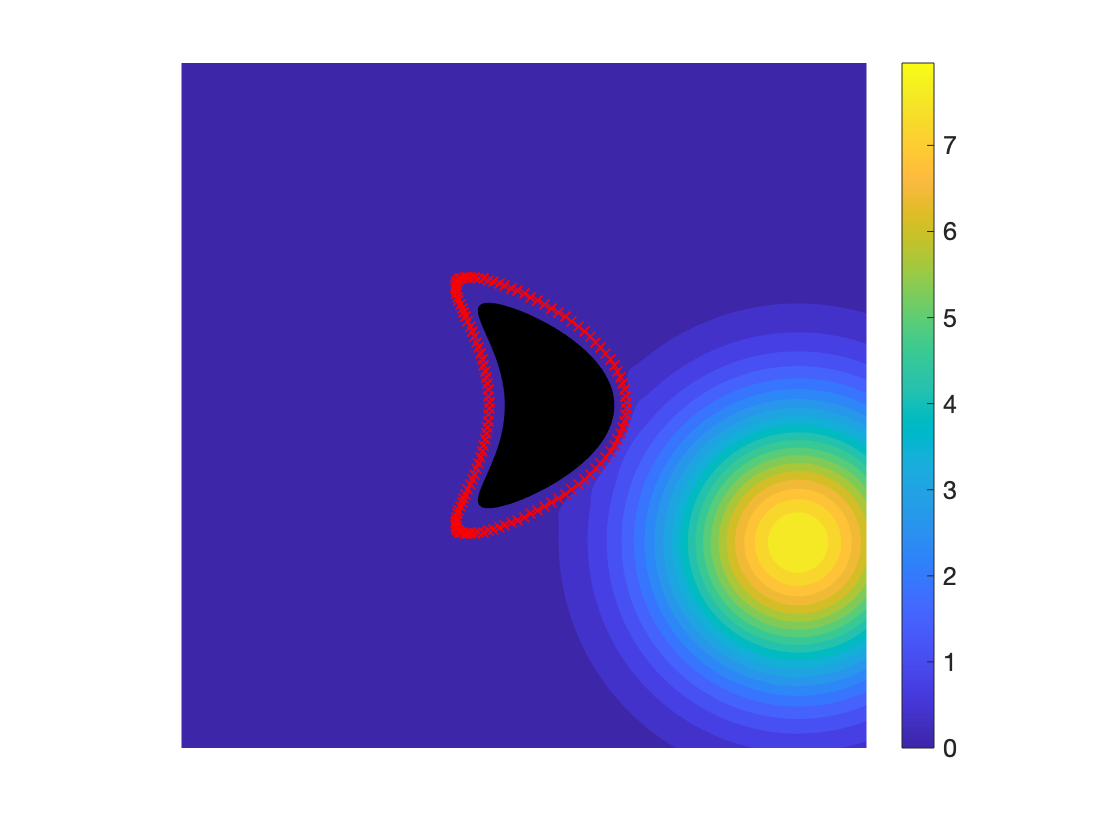} \\
    {\captionfont (a) Uncloaked object at $t=.05$s }&
    {\captionfont (b) Cloaked object at $t=.05$s }\\
    \includegraphics[width=50mm]{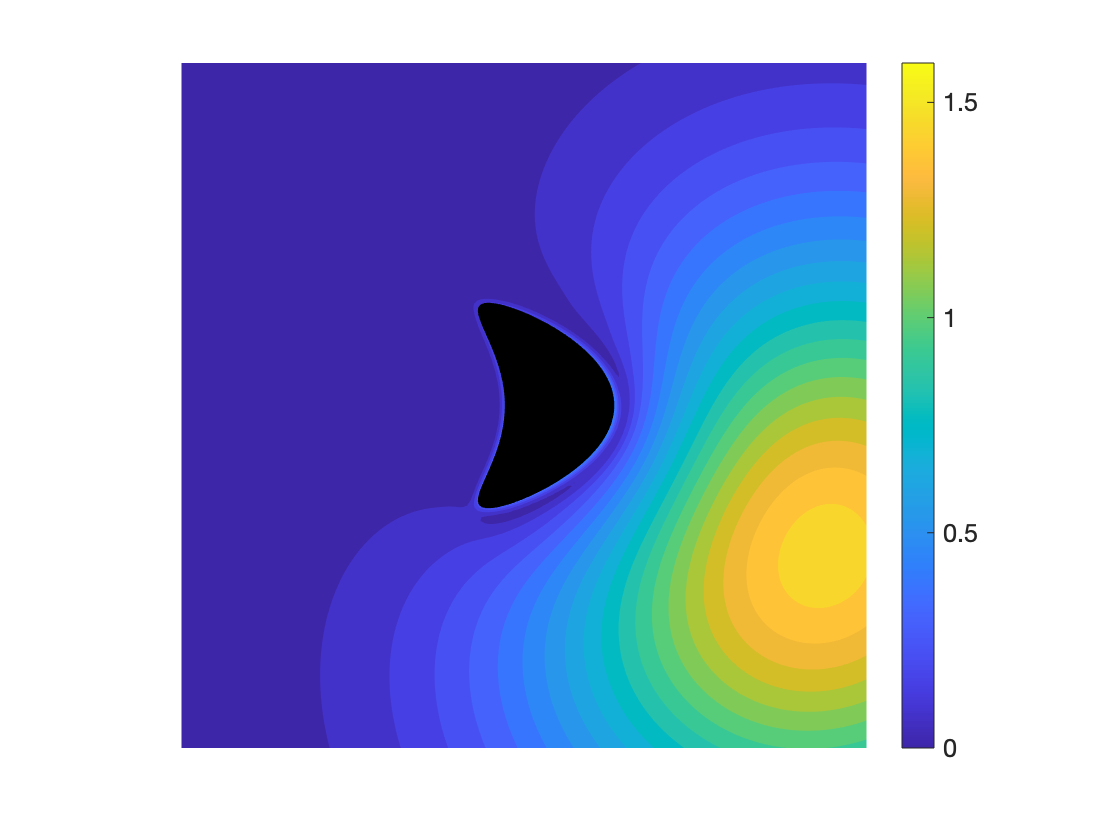} &
    \includegraphics[width=50mm]{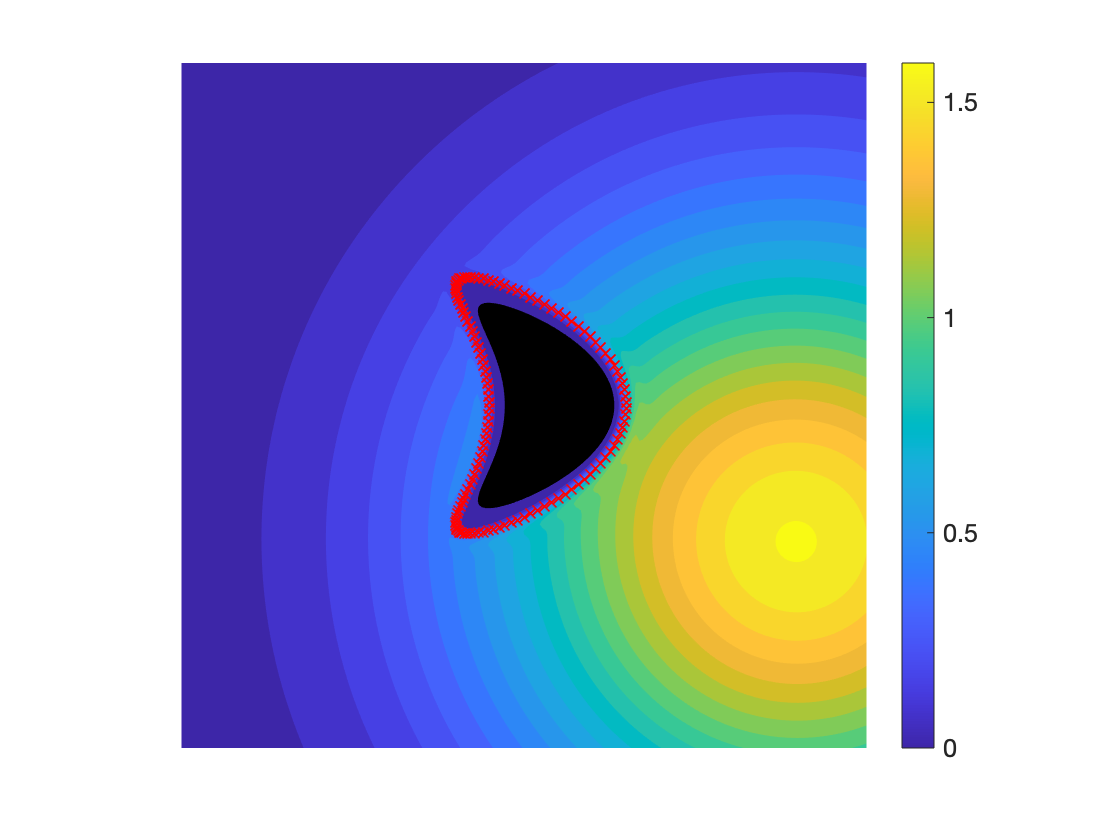} \\
    {\captionfont (c) Uncloaked object at $t=.25$s } &
    {\captionfont (d) Cloaked object at $t=.25$s } \\
    \includegraphics[width=50mm]{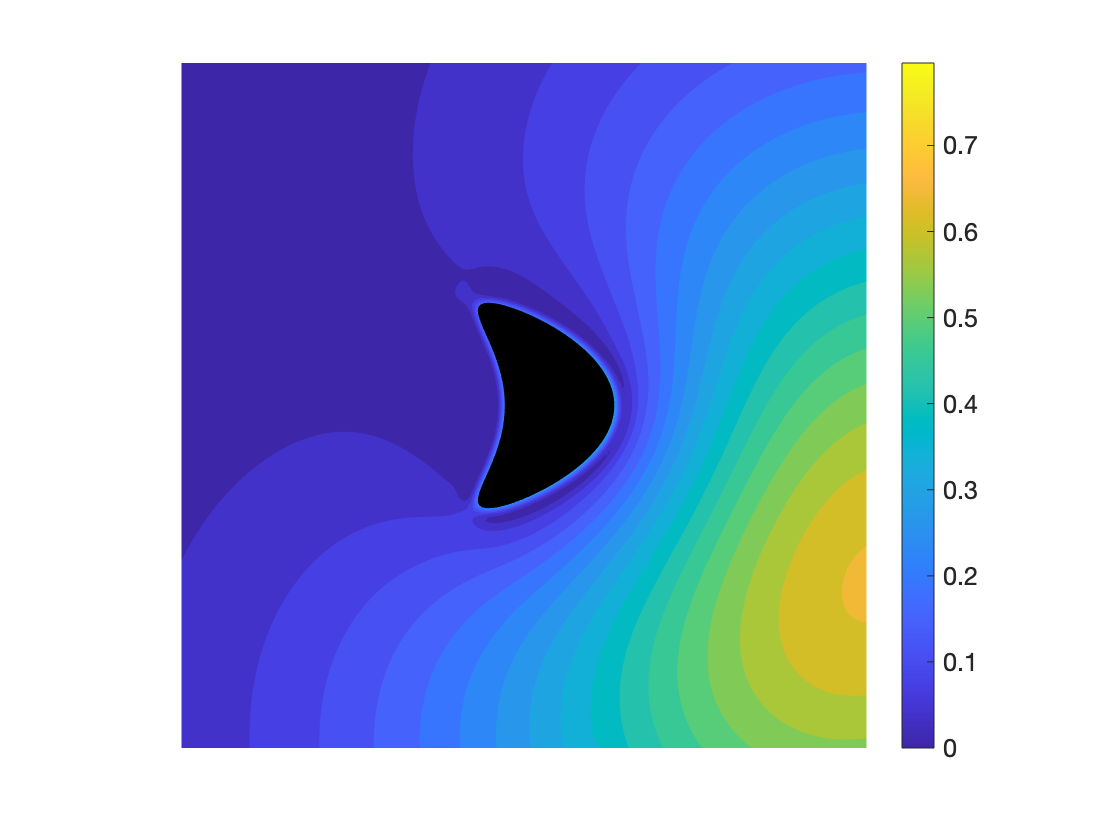} &
    \includegraphics[width=50mm]{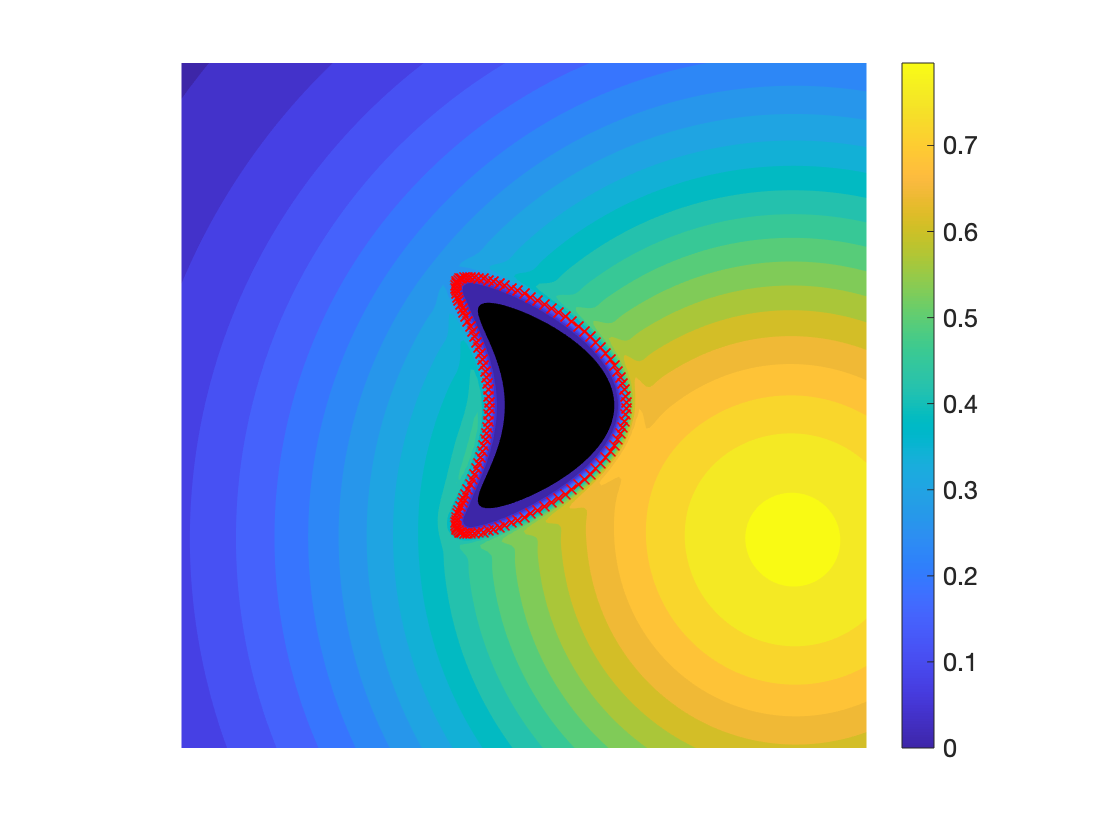} \\
    {\captionfont (e) Uncloaked object at $t=.5$s } &
    {\captionfont (f) Cloaked object at $t=.5$s }
    \end{tabular}
    \caption{Numerical example of the object cloaking problem, with object to hide having homogeneous Dirichlet boundary conditions. The incident field is produced by a point source at $x=(0.9,0.3)$ and $t=0$s. The left images (a), (c) and (e) show time snapshots of the object without the cloak and the right images (b), (d) and (f) the corresponding snapshots when the cloak is active. Placing the cloak close to the object is a challenging simulation because we expect that errors will be highest near the boundary as in \cref{fig:int_greens}. [See also movie in supplementary material]}
    \label{fig:contour_cloaking}
\end{figure}

%%%%%%%%%%%%%%%%%%%%%%%%%%%%%%%%%%%%%%%%%%%%%%%%%%%%%%%%%%%%%%%%%%%%%%%%
\section{Mimicking}
\label{sec:mimic}

Another possible application of the boundary representation formulas \eqref{eqn:brep} and \eqref{eqn:outer_brep} is to mimic sources or passive objects. This is done in two steps. First we cancel out the original source or suppress the scattering of the original object using a source distribution on a surface $\partial\Omega$ surrounding the object. Second we adjust the source distribution so that the object or source appears to the observer as another object or source. We illustrate this idea with two cases: making sources look like other sources (\cref{sec:mimic:src}) and making a passive object look like a different passive object (\cref{sec:mimic:obj}). Other combinations are possible but are not presented here.

%%%%%%%%%%%%%%%%%%%%%%%%%%%%%%%%%%%%%%%%%%%%%%%%%%%%%%%%%%%%%%%%%%%%%%%%
\subsection{Source mimicking}
\label{sec:mimic:src}
For source mimicking, we consider the problem where there is a compactly supported source distribution, $f(x,t)$ which we seek to make appear as a different compactly supported source distribution, $g(x,t)$ from thermal measurements outside of a region $\overline{\Omega}$ (where $\Omega$ is an open bounded set with Lipschitz boundary). The support of both distributions is assumed to be contained in $\overline{\Omega}$ for all $t>0$. The two corresponding solutions to the heat equation \eqref{eqn:heat_bis} are $v(x,t; f)$ and $v(x,t;g)$, and we further assume they satisfy \cref{cond:rad1}. 

Mimicking can be achieved by simultaneously canceling the field $v(x,t;f)$ outside $\overline{\Omega}$ and adding $v(x,t;g)$ also outside $\overline{\Omega}$. Both can be done using the results in \cref{sec:cloak:src}. That is, using \eqref{eqn:ext_cloak} we can find a monopole and dipole density on $\partial\Omega$ that generates a field
\begin{equation}\label{eq.cloakfield}
    v_c(x,t) = \begin{cases} 0 &x \in \Omega,\\
                    v(x,t;g) - v(x,t;f) &x \notin \overline{\Omega}.
    \end{cases}
\end{equation}
In this way the field $v(x,t;f) + v_c(x,t)$ is equal to $v(x,t;g)$ outside of $\overline{\Omega}$, as desired.

A numerical example to illustrate the method is given in \cref{fig:src_mimic_fields_cont}. Fields are calculated in $[0,1]^2$ using a uniform grid of 200 by 200 points at $t=0.2$ s.  Here a point source at $y^{(1)}=(0.6,0.4)$ and $t=0$ is made to appear as a point source at $y^{(2)}=(0.39,0.6)$ and $t=0$, from thermal measurements outside of $\overline{\Omega}$. 
\Cref{fig:src_mimic_fields_cont} (a) and (b) represent the fields $v(x,t;f)$ and $v(x,t;g)$, where $f(x,t)=\delta(x-y^{(1)},t)$ and $g(x,t)=\delta(x-y^{(2)},t)$. \Cref{fig:src_mimic_fields_cont}(c) shows the field $v(x,t;f)+ v_c(x,t)$, where $v_c$ has been constructed by applying \eqref{eq.cloakfield}.
An error plot is shown in \cref{fig:src_mimic_fields_cont}(d) where the error is largest near the boundary as we expect based on the sensitivity analysis of \cref{sec:irheq}. We believe the diagonal line with small error in \cref{fig:src_mimic_fields_cont}(d) is an artifact of our choice of sources. 

%%%%%%%%%%%%%%%%%%%%%%%%%%%%%%
\begin{figure}%[ht]
    \centering
    \begin{tabular}{cc}
        \includegraphics[width=50mm]{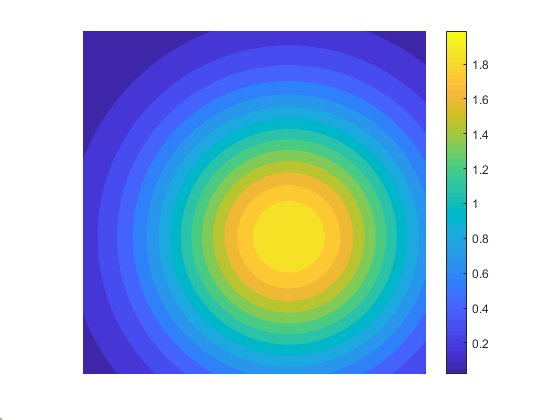} &
        \includegraphics[width=50mm]{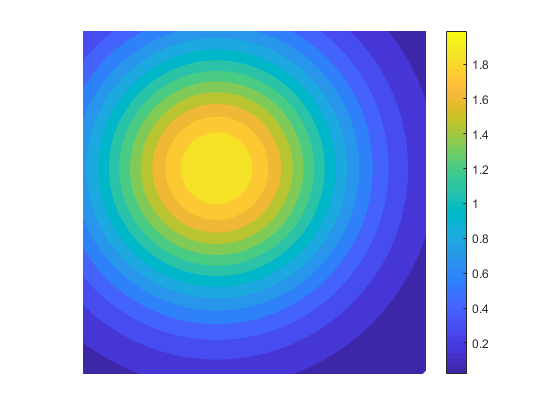}  \\
        {\captionfont(a) Original source }& {\captionfont(b) Source to mimic} \\
        \includegraphics[width=50mm]{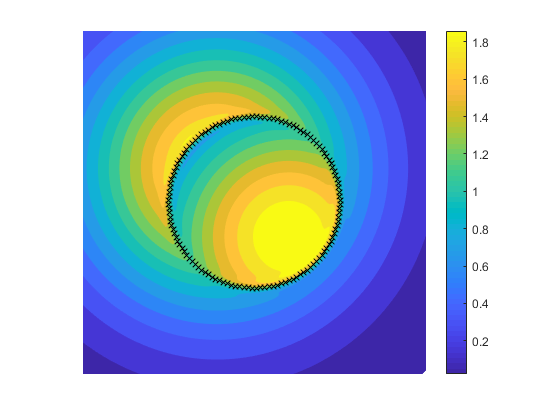} &
        \includegraphics[width = 50mm]{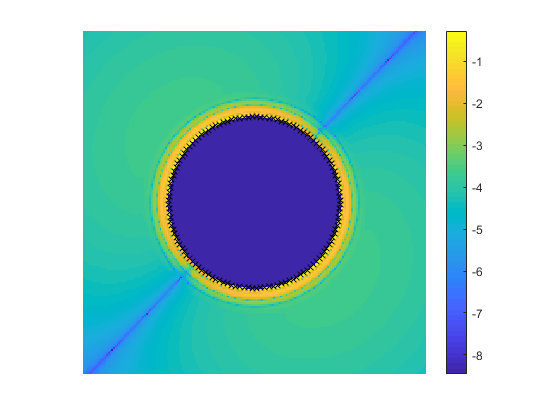}\\
        {\captionfont (c) Mimicked source} & {\captionfont (d) $\log_{10}$ plot of errors}
    \end{tabular}
    \caption{Time snapshots from a numerical example of the source mimicking problem with point sources at $t=0.2$s. In (c) the original point source in (a) is made to appear as the point source in (b) from the perspective outside of the cloaking region $\Omega$. The point source in (a) is at $x=(0.6,0.4)$ and $t=0$ and in (b) is at $x=(0.39,0.6)$ and $t=0$. A plot of the $\log_{10}$ errors in the exterior appears in (d). The diagonal line is an artifact due to the symmetry of the problem. A similar problem with point sources at $x=(0.5,0.4)$ and $x=(0.5,0.6)$ produces a horizontal line.
    }
    \label{fig:src_mimic_fields_cont}
\end{figure}

%%%%%%%%%%%%%%%%%%%%%%%%%%%%%%%%%%%%%%%%%%%%%%%%%%%%%%%%%%%%%%%%%%%%%%%%
\subsection{Mimicking a passive object} 
\label{sec:mimic:obj}
Consider a passive object $R$ that is completely contained in an open set $\Omega$ and a source exterior to $\Omega$ which produces the incident field $u_i$. The goal is to make $R$ look like a different passive object, $S$, from thermal measurements outside $\overline{\Omega}$. To achieve this we can use linearity and both the interior and exterior boundary representation formulas to find monopole and dipole densities on $\partial \Omega$ producing a field
\begin{equation}
    u_c(x,t) = \begin{cases}
     -u_i(x,t), & x \in \Omega,\\
     v_s(x,t), & x \notin \overline{\Omega},
    \end{cases}
\end{equation}
where $v_s(x,t)$ is the scattered field corresponding to the object $S$, included also in $\Omega$, resulting from the incident field $u_i(x,t)$. In this fashion the total field is $u_i(x,t) + v_s(x,t)$ outside of $\overline{\Omega}$ and $0$ inside of $\Omega$.  Its associated scattered field is  $v_s(x,t)$ outside of $\overline{\Omega}$ and $0$ inside of $\Omega\setminus{R}$ as desired.  

To illustrate the method we consider a ``kite'' object with homogeneous Dirichlet boundary conditions and make it appear as a ``flower'' object with identical boundary conditions.  Here the field $u_i$ is generated by a point source at $x=(0.25,0.5)$ and $t=0$. For the heat equation we took $k=0.2$ and  the domain is $\Omega = B(x_0,r)$ with $x_0 = (0.5,0.5)$ and $r=0.25$. We computed the fields on the unit square $[0,1]^2$ with a $200\times 200$ uniform grid. The field $u_c$ is found by approximating the integral \eqref{eqn:brep} using the midpoint rule in time with 180 equal length subintervals of $[0,0.05]$  and the trapezoidal rule on $\partial \Omega$ with 128 uniformly spaced points. A more detailed explanation, including how the scattered fields are calculated, is in \cref{sec:numerics}. \Cref{fig:conotur_mimic_fields}(a) and (b) show the scattered fields from two different objects and \cref{fig:conotur_mimic_fields}(c) shows the mimicked scattered field. Because we are approximating the fields numerically, the field $u_i + u_c$ is very close to zero in $\Omega$, but not exactly zero. The errors (which are not reported here) are larger near the boundary and decay as we move outwards, as we observed in \cref{sec:irheq}.

%%%%%%%%%%%%%%%%%%%%%%%%%%%%%%
\begin{figure}%[ht]
    \centering
    \begin{tabular}{p{0.33\textwidth}p{0.33\textwidth}p{0.33\textwidth}}
    \includegraphics[width=0.4\textwidth]{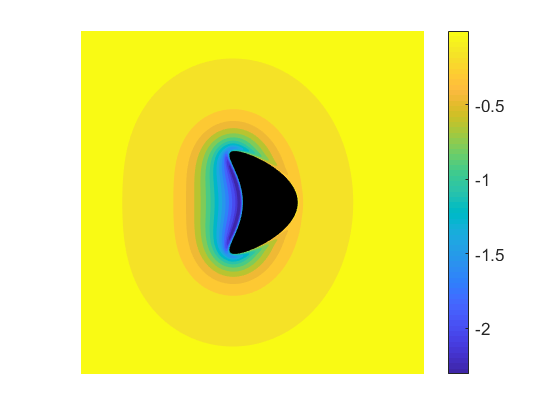} & \includegraphics[width=0.4\textwidth]{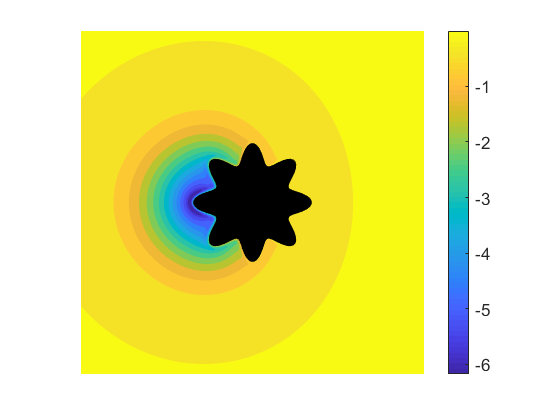} & \includegraphics[width=0.4\textwidth ]{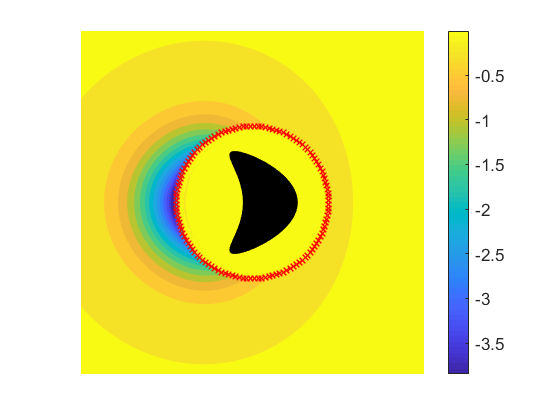} \\
    {\captionfont (a) Scattered field from a kite} &  {\captionfont (b) Scattered field from a flower} &{\captionfont (c) Mimicked scattered field }
    \end{tabular}
    \caption{Numerical example of the object mimicking problem with a ``kite'' and ``flower" object, both with homogenous Dirichlet boundary conditions. A snapshot the scattered field from the original object at time $t=0.05$s appears in (a). In (c), the scattered field from the original object is made to appear as the scattered field from the object in (b) from the perspective of thermal measurements outside of the cloaking region $\overline{\Omega}$.}
    
    \label{fig:conotur_mimic_fields}
\end{figure}

\begin{remark}
Although we have only shown mimicking of passive objects, the same techniques outlined in \cref{rmk:other}, could be used to suppress the field generated by an active object, which is part of what needs to be done to mimic an active object.
\end{remark}

%%%%%%%%%%%%%%%%%%%%%%%%%%%%%%%%%%%%%%%%%%%%%%%%%%%%%%%%%%%%%%%%%%%%%%%%
\section{A simple numerical approach to potential theory for the heat equation}
\label{sec:numerics}

The boundary representation formulas \eqref{eqn:brep} and \eqref{eqn:outer_brep} can expressed very efficiently in terms of the single and double layer potentials for the heat equation defined for $t>0$ by
\[
\begin{aligned}
\mathcal{K}_0(\psi)(x,t)&=\int_0^t ds\int_{\partial \Omega} dS(y) [\psi (y,s)K(x-y,t-s)], \; x \notin \partial \Omega,\\
\mathcal{K}_1(\varphi)(x,t)&=\int_0^t ds\int_{\partial \Omega} dS(y) [\varphi (y,s)\frac{\partial K}{\partial n}(x-y,t-s)], \; x \notin \partial \Omega,
\end{aligned}
\]
as well as the corresponding boundary layer operators given for $t>0$ by
\[
\begin{aligned}
\mathcal{V}(\psi)(x,t)&=\int_0^t ds\int_{\partial \Omega} dS(y) [\psi (y,s)K(x-y,t-s)], \; x \in \partial \Omega , \\
\mathcal{K}(\varphi)(x,t)&=\int_0^t ds\int_{\partial \Omega} dS(y) [\varphi (y,s)\frac{\partial K}{\partial n}(x-y,t-s)]\; x \in \partial \Omega,
\end{aligned}
\]
where $\phi(x,t)$ and $\psi(x,t)$ are time dependent densities defined on $\partial\Omega$. A full derivation of these and other potential theory operators for the heat equation and their properties can be found in  \cite{Costabel:1990:BIO,Kress:2014:LIE,Costabel:2017:TPB}. For a review of classic potential theory see e.g.  \cite{Costabel:1988:BIO,Colton:2013:IAE,Kress:2014:LIE,Ammari:2009:LPT}.

With the potential theory notation, the interior boundary representation formula \eqref{eqn:brep}, with zero initial condition, becomes
\begin{equation}\label{eqn:brep:op}
    u_\Omega(x,t) = \mathcal{K}_0\Big(\frac{\partial u}{ \partial n}\Big)(x,t)-\mathcal{K}_1(u)(x,t).
\end{equation}
Whereas the exterior boundary representation formula \eqref{eqn:outer_brep} becomes
\begin{equation}\label{eqn:outer_brep:op}
    v_\Omega(x,t) = \mathcal{K}_1(v)(x,t)-\mathcal{K}_0\Big(\frac{\partial v}{\partial n}\Big)(x,t).
\end{equation}

Galerkin methods are commonly used to approximate the spatial integrals in \cref{eqn:outer_brep:op,eqn:brep:op}, with a number of different approaches to deal with the integration in time. For instance time marching \cite{Pina:1984:ATH}, time-space Galerkin methods \cite{Noon:1988:SLH,Dohr:2019:DPS} convolution quadrature \cite{Qiu:2019:TDB} and collocation \cite{Hamina:1994:SCM}. For simplicity we opted for an approach based on the trapezoidal rule for the integration on $\partial\Omega$ and the midpoint rule for the time convolution.
To be more precise, there are two convolutions that need to be evaluated in order to calculate the boundary representations in \eqref{eqn:brep:op} and \eqref{eqn:outer_brep:op}; a convolution in space and a convolution in time. 
For the spatial integration, trapezoidal rule is used with uniformly placed points on a parametric representation of $\partial \Omega$. Since this amounts to integrating a periodic function, we can expect that the convergence rate of the trapezoidal rule depends explicitly on the rate of decay of the Fourier coefficients of the function to be integrated \cite{Hyde:2003:FHO}. Due to the smoothness of the heat kernel, exponential convergence is expected. For the integration in time, the convolutions are of the form
\[
 \int_0^t ds\;g(s) f(t-s).
\] 
These convolutions are approximated with the midpoint rule as follows 
\begin{equation}
 \int_0^{j\Delta t} ds\; g(s) f(j\Delta t-s) \approx \Delta t \sum_{k=1}^j g( (k-\frac{1}{2})\Delta t) f((j-k)\Delta t + \frac{1}{2}\Delta t),
 \label{eq:midpoint}
\end{equation}
which avoids evaluating $f$ at $t=0$. This is handy in our case because of the singularity of the heat kernel at $(x,t)=(0,0)$, which only occurs for boundary layer operators. After this space-time discretization, the resulting approximations are by construction an interpolation on the space-time grid of a finite distribution of monopoles and dipoles located on $\partial \Omega$. These distributions are smooth solutions of the homogeneous heat equation on any open set of $(\real^d \setminus \partial \Omega) \times \real$. Moreover \cref{lem:heatkernel} (see eq. \eqref{eq:bounderiv2}) ensures that they are bounded for $t>0$ and $x\notin B(0,r)$ for $r$  large enough. Thus, as they satisfy  \cref{eq.bounduniqueness},  we can apply the maximum principle over any finite time window $[t_1,t_2]$ (with $0\leq t_1<t_2$) on any closed set of $\real^d \setminus \overline{\Omega}$ that does not intersect $\partial\Omega$.
We leave an accuracy study of the numerical approximation we use to future work. In particular there are more accurate ways of dealing with the approximation for $s\in[0,\Delta t]$ than the midpoint rule we use, see e.g. \cite{Pina:1984:ATH,Noon:1988:SLH,Hamina:1994:SCM,Tausch:2007:FMS}.

%%%%%%%%%%%%%%%%%%%%%%%%%%%%%%%%%%%%%%%%%%%%%%%%%%%%%%%%%%%%%%%%%%%%%%%%
\subsection{Scattered field computation}
In the case of inclusions, finding the scattered field requires the use of the boundary layer operators. For simplicity we only consider a homogeneous Dirichlet inclusion, $R$, i.e. where the temperature on $\partial R$ is held constant at $0$. Neumann inclusions, and inclusions with varying thermal diffusivity, require the introduction of other boundary integral operators (adjoint double layer and hypersingular \cite{Costabel:1990:BIO}), but a similar numerical approach can be applied to this case.

The  idea is to look for a scattered field $u_s$ of the form \eqref{eqn:outer_brep:op} outside of $R$. 
Since the temperature at $\partial R$ is constant and equal to $0$, the scattered field satisfies $u_s|_{\partial R} = -u_i|_{\partial R}$. Hence we know the Dirichlet data on $\partial R$ in the representation formula \eqref{eqn:outer_brep:op}, but not the Neumann data. We treat this as an unknown boundary density $\psi$ in
\begin{equation}
 \label{eq:us_out}
 u_s(x,t) = \mathcal{K}_1 (-u_i|_{\partial R})(x,t) - \mathcal{K}_0 (\psi)(x,t).
\end{equation}
The latter formula is only valid for $x \notin  \overline{R}$. To obtain an integral equation on $\partial R$ we take the limit of \eqref{eq:us_out} as $x$ approaches $\partial R$. The limit could be different if we approach $\partial R$ from the inside or from the outside. The limits are given by the so-called jump relations. For $x \in \partial R$, the jump relation needed for the single layer potential is 
\begin{equation}
        \lim_{z \rightarrow x} \mathcal{K}_0 \psi(z,t) =\mathcal{V}\psi(x,t),
    \label{eqn:sl_jump}
\end{equation}
which holds for $z$ tending towards $\partial R$ from both the interior and exterior of $R$.
The jump relations needed for the double layer potential are
\begin{equation}
        \lim_{z \rightarrow x^+} \mathcal{K}_1 \varphi(z,t) = \frac{1}{2}\varphi(x,t)+\mathcal{K}\varphi(x,t),~\text{and}
        \lim_{z \rightarrow x^-} \mathcal{K}_1 \varphi(z,t) = \frac{-1}{2}\varphi(x,t)+\mathcal{K}\varphi(x,t),
    \label{eqn:dl_jump}
\end{equation}
where $z\to x^+$ denotes approaching $\partial R$ from the exterior of $R$ and $z\to x^-$ from the interior. Using these jump relations and \eqref{eq:us_out} we have 
\begin{equation}
\begin{aligned}
    \lim_{z \rightarrow x^+}u_s(z,t) &= \lim_{z \rightarrow x^+}\mathcal{K}_1(-u_i(z,t))- \lim_{z \rightarrow x^+}\mathcal{K}_0\psi(z,t) \\ \Rightarrow
    -u_i|_{\partial R} &= \frac{-1}{2}u_i|_{\partial R} +\mathcal{K}(-u_i|_{\partial R})-\mathcal{V}\psi.\\
\end{aligned}
\end{equation}
Rearranging terms yields a boundary integral equation for $\psi$
\begin{equation}
     \mathcal{V}\psi= \frac{u_i|_{\partial R}}{2}+\mathcal{K}(-u_i|_{\partial R}).
\label{eqn:calderon_first_row}
\end{equation}
We discretize \eqref{eqn:calderon_first_row} as a linear system where the unknown is $\psi$ evaluated on a uniform grid of $\partial R$ and of $[0,T]$. The boundary integral operators $\mathcal{V}$ and $\mathcal{K}$ in \eqref{eqn:calderon_first_row} are discretized  using the trapezoidal rule in space and the midpoint rule in time. For instance, $\mathcal{V}$ is approximated by a $MN \times MN$ matrix which is lower triangular by blocks, with each block of size $N \times N$. Here  $M$ is the number of time steps and $\partial R$ is approximated by a polygon with $N$ sides:
\begin{equation}
\mathcal{V} \approx
\left[
\begin{array}{cccc}
V_\frac{1}{2} &  &  &  \\[1.2em]
V_\frac{3}{2} & V_\frac{1}{2} &  &  \\
  \vdots &  & \ddots &  \\
V_{M-\frac{1}{2}} & V_{M-\frac{3}{2}} &  \cdots & V_\frac{1}{2}
\end{array}
\right].
\label{eq:vdiscr}
\end{equation}
The $V_i$ are $N \times N$ matrices with entries $(V_i)_{jk} = \ell_k K(x_j - x_k,i\Delta t)$, where $x_k$ is the center of the $k-$th segment of length $\ell_k$. Clearly the matrix in \eqref{eq:vdiscr} is guaranteed to be invertible if $V_{\frac{1}{2}}$ is invertible.
Though we have not studied the invertibility of $V_{\frac{1}{2}}$, we observe that it may become singular for a given spatial discretization if the temporal discretization is not fine enough. 

%%%%%%%%%%%%%%%%%%%%%%%%%%%%%%%%%%%%%%%%%%%%%%%%%%%%%%%%%%%%%%%%%%%%%%%%
\section{Summary and perspectives}
\label{sec:summary}

We proposed a strategy for active cloaking for the time-dependent parabolic heat (or mass, or diffusive light) equation.  Similar to previous work for active cloaking for e.g. the Helmholtz or Laplace equation (e.g. modelling time-harmonic waves and thermostatic problems), our results rely on active sources coming from Green identities to reproduce solutions inside or outside a bounded domain. The idea is to use a source distribution on a closed surface to reproduce certain solutions to the heat equation inside the surface and the zero solution outside or vice versa. We give a growth condition which is sufficient to guarantee that a solution to the heat equation can be reproduced outside of a closed surface. We apply these theoretical results in four ways:  interior cloaking of a source, interior cloaking of an object, source mimicking, and object mimicking. For the cloaking problems the idea is to find an active surface that surrounds the object or source to make the object or source undetectable by thermal measurements outside the surface. In the mimicking problems, instead of making the object undetectable, we make the source or object appear as a different source or object from the perspective of thermal measurements outside the cloak. Our solution to these problems inherits the limitations of the reproduction method, namely that the fields must be known for all time on the active surface surrounding  the object or source we want to cloak or mimic. 
However the maximum principle guarantees some stability of our approach, as it is based on boundary representation formulas. This is a special feature of the heat equation. We illustrate our method with simple potential theory based simulations which are consistent with the heat equation and thus allow us to interpret the numerical errors using the maximum principle.
Our study was limited to the case where the initial condition is zero or harmonic. It may be possible to use \eqref{eqn:initial} to cancel out the initial condition with a boundary integral, but (a) it is not clear whether this mathematical construct has a physical interpretation and (b) this representation formula is only valid inside a bounded domain. These are questions we plan on exploring.
Although we focused on the isotropic heat equation, it may be possible to carry out a similar cloaking strategy for anisotropic media and when an advection term is added to the heat equation. Our approach could also be tied to the active cloaking strategies for the Helmholtz equation in \cite{Guevara:2009:AEC,Guevara:2009:BEO} by going to Fourier or Laplace domain in time, so that the active sources do not completely surround the object to achieve partial cloaking. For the Fourier domain, the physical interpretation is to study time-harmonic sources. Another open question is whether the growth condition we provided is optimal and how it is related to the uniqueness question for the exterior problem associated with different boundary condition types for the heat equation. 

We note that this work could be also adapted to cloaking \cite{Zhang:2008:CMM,Greenleaf:2008:CMM} and mimicking \cite{Diatta:2010:NSCM} of quantum matter waves. We believe this approach could be further generalized to the Fokker-Plank equation arising in statistical mechanics, which could applied to  gravitational systems for which some passive cloaking theory has been proposed \cite{Smerlak:2012:DCST}.
Finally, as some scattering cancellation based cloaking approach has been proposed for Maxwell-Cataneo heat waves \cite{Farhat:2019:MCHW}, we believe that our method for active cloaking could be also applied to such pseudo-waves.

%\enlargethispage{20pt}

% \ethics{Insert ethics text here.} 
{\bf Data Access.} The Matlab code to reproduce the figures \cref{fig:int_greens,fig:ext_greens,fig:spatial_errors,fig:temporal_errors,fig:errors,fig:contour_cloaking,fig:src_mimic_fields_cont,fig:conotur_mimic_fields} is available in repository [TBA]. For example, \cref{fig:int_greens} can be reproduced by running the Matlab script {\tt fig3.m}. The scripts for \cref{fig:int_greens,fig:ext_greens} generate additional plots that are mentioned in the respective figure captions. \Cref{fig:contour_cloaking} is animated in {\tt movie1.mp4}. The spatial distribution for the monopole/dipole added noise (\cref{sec:irheq:num}) can be generated with the  scripts {\tt fig9supint.m} and {\tt fig9supext.m}.

{\bf Author Contributions.} SG suggested the problem and applications to other physical phenomena. MC, TD, and FGV proved the exterior representation formula and the other theoretical results. All authors designed the numerical experiments. TD and FGV performed the numerical experiments. All authors contributed to the writing.

%\competing{Insert competing text here.}

{\bf Funding.} TD and FGV were supported by National Science Foundation Grant DMS-2008610.

{\bf Acknowledgemnts.} FGV acknowledges support from the Fresnel institute for a visit in January 2018 during which the idea for this article first came up. FGV also thanks the Jean Kuntzmann Laboratory for hosting the author during the 2017-2018 academic year. Last but not least, the authors acknowledge Graeme Milton for the exterior cloaking idea that inspired this work.

%\disclaimer{Insert disclaimer text here.}

%%%%%%%%%% Insert bibliography here %%%%%%%%%%%%%%
\bibliographystyle{abbrv}
\bibliography{thermal}

\begin{thebibliography}{99}

\bibitem{Nguyen:2015:ATC}
Nguyen DM, Xu H, Zhang Y, Zhang B. 2015  Active thermal cloak. {\em Applied
  Physics Letters} \textbf{107}, 121901.

\bibitem{DiSalvo:1999:TCP}
DiSalvo FJ. 1999  Thermoelectric Cooling and Power Generation. {\em Science}
  \textbf{285}, 703--706.

\bibitem{Xu:2019:DAT}
Xu L, Yang S, Huang J. 2019  Dipole-assisted thermotics: Experimental
  demonstration of dipole-driven thermal invisibility. {\em Phys. Rev. E}
  \textbf{100}, 062108.

\bibitem{Peralta:2019:CHM}
Peralta I, Fachinotti V, Álvarez Hostos J. 2019  A Brief Review on Thermal
  Metamaterials for Cloaking and Heat Flux Manipulation. {\em Advanced
  Engineering Materials} \textbf{22}, 1901034.

\bibitem{Guenneau:2013:FSL}
Guenneau S, Puvirajesinghe TM. 2013  Fick's second law transformed: one path to
  cloaking in mass diffusion. {\em Journal of The Royal Society Interface}
  \textbf{10}, 20130106.

\bibitem{Puvirajesinghe:2017:FSL}
Puvirajesinghe TM, Zhi ZL, Craster R, Guenneau S. 2018  Tailoring drug release
  rates in hydrogel-based therapeutic delivery applications using graphene
  oxide. {\em Journal of The Royal Society Interface} \textbf{15}, 20170949.

\bibitem{Zeng:2013:CID}
Zeng L, Song R. 2013  Controlling chloride ions diffusion in concrete. {\em
  Scientific Reports} \textbf{3}, 3359.

\bibitem{Schittny:2014:DLSM}
Schittny R, Kadic M, Buckmann T, Wegener M. 2014  Invisibility cloaking in a
  diffusive light scattering medium. {\em Science} \textbf{345}, 427--429.

\bibitem{Guenneau:2012:TTC}
Guenneau S, Amra C, Veynante D. 2012  Transformation thermodynamics: cloaking
  and concentrating heat flux. {\em Opt. Express} \textbf{20}, 8207--8218.

\bibitem{Ma:2013:TTC}
Ma Y, Lan L, Jiang W, Sun F, He S. 2013  A transient thermal cloak
  experimentally realized through a rescaled diffusion equation with
  anisotropic thermal diffusivity. {\em NPG Asia Materials} \textbf{5},
  e73--e73.

\bibitem{Craster:2018:CMH}
Craster RV, Guenneau SRL, Hutridurga HR, Pavliotis GA. 2018  Cloaking via
  mapping for the heat equation. {\em Multiscale Model. Simul.} \textbf{16},
  1146--1174.

\bibitem{Schittny:2013:ETT}
Schittny R, Kadic M, Guenneau S, Wegener M. 2013  Experiments on Transformation
  Thermodynamics: Molding the Flow of Heat. {\em Phys. Rev. Lett.}
  \textbf{110}, 195901.

\bibitem{Petiteau:2014:SEETC}
Petiteau D, Guenneau S, Bellieud M, Zerrad M, Amra C. 2015  Spectral
  effectiveness of engineered thermal cloaks in the frequency regime. {\em
  Scientific reports} \textbf{4}, 7386.

\bibitem{Miller:2005:OPC}
Miller DAB. 2006  On perfect cloaking. {\em Opt. Express} \textbf{14},
  12457--12466.

\bibitem{Ffowcs:1984:RLA}
Ffowcs~Williams JE. 1984  Review Lecture: Anti-Sound. {\em Proc. R. Soc. A}
  \textbf{395}, 63--88.

\bibitem{Elliot:1991:ACO}
Nelson P, Elliot SJ. 1991 {\em Active Control of Sounds}.
Academic Press, New York first edition.

\bibitem{Guevara:2009:BEO}
Guevara~Vasquez F, Milton GW, Onofrei D. 2009a  Broadband exterior cloaking.
  {\em Opt. Express} \textbf{17}, 14800--14805.

\bibitem{Guevara:2009:AEC}
Guevara~Vasquez F, Milton GW, Onofrei D. 2009b  Active Exterior Cloaking for
  the 2D {L}aplace and {H}elmholtz Equations. {\em Phys. Rev. Lett.}
  \textbf{103}, 073901.

\bibitem{Guevara:2011:ECA}
Guevara~Vasquez F, Milton GW, Onofrei D. 2011  Exterior cloaking with active
  sources in two dimensional acoustics. {\em Wave Motion} \textbf{48},
  515--524.
Special Issue on Cloaking of Wave Motion.

\bibitem{Norris:2012:SAF}
Norris AN, Amirkulova FA, Parnell WJ. 2012  Source amplitudes for active
  exterior cloaking. {\em Inverse Problems} \textbf{28}, 105002.

\bibitem{Guevara:2012:MAT}
Guevara~Vasquez F, Milton GW, Onofrei D. 2012  Mathematical analysis of the two
  dimensional active exterior cloaking in the quasistatic regime. {\em Analysis
  and Mathematical Physics} \textbf{2}, 231--246.

\bibitem{Guevara:2013:TEA}
Guevara~Vasquez F, Milton GW, Onofrei D, Seppecher P. 2013  Transformation
  elastodynamics and active exterior cloaking. In Craster RV, Guenneau S,
  editors, {\em Acoustic metamaterials: Negative refraction, imaging, lensing
  and cloaking}. Springer.

\bibitem{Norris:2014:AEC}
Norris AN, Amirkulova FA, Parnell WJ. 2014  Active elastodynamic cloaking. {\em
  Math. Mech. Solids} \textbf{19}, 603--625.

\bibitem{O'Neill:2015:ACO}
O'Neill J, Selsil O, McPhedran RC, Movchan AB, Movchan NV. 2015  Active
  cloaking of inclusions for flexural waves in thin elastic plates. {\em Quart.
  J. Mech. Appl. Math.} \textbf{68}, 263--288.

\bibitem{McPhedran:2016:ACO}
O'Neill J, Selsil O, McPhedran RC, Movchan AB, Movchan NV, Henderson~Moggach C.
  2016  Active cloaking of resonant coated inclusions for waves in membranes
  and {K}irchhoff plates. {\em Quart. J. Mech. Appl. Math.} \textbf{69},
  115--159.

\bibitem{Avanzini:2020:CC}
Avanzini F, Falasco G, Esposito M. 2020  Chemical cloaking. {\em Phys. Rev. E}
  \textbf{101}, 060102.

\bibitem{MA:2013:EOA}
Ma Q, Mei ZL, Zhu SK, Jin TY, Cui TJ. 2013  Experiments on Active Cloaking and
  Illusion for {L}aplace Equation. {\em Phys. Rev. Lett.} \textbf{111}, 173901.

\bibitem{Ma:2016:OAC}
Ma Q, Yang F, Jin TY, Mei ZL, Cui TJ. 2016  Open active cloaking and illusion
  devices for the Laplace equation. {\em Journal of Optics} \textbf{18},
  044004.

\bibitem{Selvanayagam:2013:EDO}
Selvanayagam M, Eleftheriades GV. 2013  Experimental Demonstration of Active
  Electromagnetic Cloaking. {\em Phys. Rev. X} \textbf{3}, 041011.

\bibitem{Lai:2009:IOT}
Lai Y, Ng J, Chen H, Han D, Xiao J, Zhang ZQ, Chan CT. 2009  Illusion Optics:
  The Optical Transformation of an Object into Another Object. {\em Phys. Rev.
  Lett.} \textbf{102}, 253902.

\bibitem{Zheng:2010:EOC}
Zheng HH, Xiao JJ, Lai Y, Chan CT. 2010  Exterior optical cloaking and
  illusions by using active sources: A boundary element perspective. {\em Phys.
  Rev. B} \textbf{81}, 195116.

\bibitem{Lions:1973:PDE}
Lions JL, Magenes E. 1973 {\em Non-homogenous Boundary Values Problem and
  Applications} vol. III.
Springer-Verlag.

\bibitem{Pina:1984:ATH}
Pina HLG, Fernandes JLM. 1984  Applications in transient heat conduction. In
  {\em Topics in boundary element research, {V}ol. 1} pp. 41--58. Springer,
  Berlin.

\bibitem{Costabel:1990:BIO}
Costabel M. 1990  Boundary integral operators for the heat equation. {\em
  Integral Equations Operator Theory} \textbf{13}, 498--552.

\bibitem{Friedman:1964:PDE}
Friedman A. 1964 {\em Partial differential equations of parabolic type}.
Prentice-Hall, Inc., Englewood Cliffs, N.J.

\bibitem{NIST:DLMF}
{\relax DLMF} {\it NIST Digital Library of Mathematical Functions}.
  http://dlmf.nist.gov/, Release 1.0.21 of 2018-12-15.
F.~W.~J. Olver, A.~B. {Olde Daalhuis}, D.~W. Lozier, B.~I. Schneider, R.~F.
  Boisvert, C.~W. Clark, B.~R. Miller and B.~V. Saunders, eds.

\bibitem{Brezis:1964:PDE}
Brezis H. 2011 {\em Functional analysis, Sobolev spaces and partial
  differential equations}.
Springer.

\bibitem{Evans:2010:PDE}
Evans LC. 2010 {\em Partial differential equations} vol. {19, }{\em Graduate
  Studies in Mathematics}.
American Mathematical Society second edition.

\bibitem{Kress:2014:LIE}
Kress R. 2014 {\em Linear integral equations.} vol. {82, }{\em Applied
  Mathematical Sciences.}
Springer third edition.

\bibitem{Tychonoff:1935:TDP}
Tychonoff A. 1935  Th\'eor\`emes d'unicit\'e pour l'\'equation de la chaleur.
  {\em Mat. Sb.} \textbf{42}, 199--216.

\bibitem{Arnold:1987:BIE}
Arnold DN, Noon PJ. 1987  Boundary integral equations of the first kind for the
  heat equation. In {\em Boundary elements {IX}, {V}ol. 3 ({S}tuttgart, 1987)}
  pp. 213--229. Comput. Mech., Southampton.

\bibitem{Noon:1988:SLH}
Noon PJ. 1988 {\em The single layer heat potential and {G}alerkin boundary
  element methods for the heat equation}.
PhD thesis.
Thesis (Ph.D.)--University of Maryland, College Park.

\bibitem{Qiu:2019:TDB}
Qiu T, Rieder A, Sayas FJ, Zhang S. 2019  Time-domain boundary integral
  equation modeling of heat transmission problems. {\em Numer. Math.}
  \textbf{143}, 223--259.

\bibitem{Dohr:2019:DPS}
Dohr S. 2019 {\em Distributed and Preconditioned Space–Time Boundary Element
  Methods for the Heat Equation}.
PhD thesis Graz University of Technology.

\bibitem{Colton:2013:IAE}
Colton D, Kress R. 2013 {\em Inverse acoustic and electromagnetic scattering
  theory} vol. {93, }{\em Applied Mathematical Sciences}.
Springer, New York third edition.

\bibitem{Lions:1972:PDE}
Lions JL, Magenes E. 1972 {\em Non-homogenous Boundary Values Problem and
  Applications} vol.~II.
Springer-Verlag.

\bibitem{Ammari:2005:DATI}
Ammari H, Iakovleva E, Kang H, Kim K. 2005  Direct Algorithms for Thermal
  Imaging of Small Inclusions. {\em Multiscale Model. Simul.} \textbf{4},
  1116--1136.

\bibitem{Hohage:2005:NSH}
Hohage T, Sayas F. 2005  Numerical solution of a heat diffusion problem by
  boundary element methods using the Laplace transform. {\em Numerische
  Mathematik} \textbf{102}, 67--92.

\bibitem{Costabel:2017:TPB}
Costabel M, Sayas F. 2017  Time‐dependent problems with the boundary integral
  equation method. {\em Encyclopedia of Computational Mechanics Second Edition}
  pp. 1--24.

\bibitem{Costabel:1988:BIO}
Costabel M. 1988  Boundary integral operators on {L}ipschitz domains:
  elementary results. {\em SIAM J. Math. Anal.} \textbf{19}, 613--626.

\bibitem{Ammari:2009:LPT}
Ammari H, Kang H, Lee H. 2009 {\em Layer potential techniques in spectral
  analysis} vol. 153{\em Mathematical Surveys and Monographs}.
American Mathematical Society, Providence, RI.

\bibitem{Hamina:1994:SCM}
Martti H, Saranen J. 1994  On the spline collocation method for the
  single-layer heat operator equation. {\em Mathematics of Computation}
  \textbf{62}, 41--64.

\bibitem{Hyde:2003:FHO}
Hyde EM. 2003 {\em Fast, high-order methods for scattering by inhomogeneous
  media}.
ProQuest LLC, Ann Arbor, MI.
Thesis (Ph.D.)--California Institute of Technology.

\bibitem{Tausch:2007:FMS}
Tausch J. 2007  A fast method for solving the heat equation by layer
  potentials. {\em J. Comput. Phys.} \textbf{224}, 956--969.

\bibitem{Zhang:2008:CMM}
Zhang S, Genov D, Sun C, Zhang X. 2008  Cloaking of matter waves. {\em Phys.
  Rev. Lett.} \textbf{100}, 123002.

\bibitem{Greenleaf:2008:CMM}
Greenleaf A, Kurylev Y, Lassas M, Uhlmann G. 2008  Isotropic transformation
  optics: approximate acoustic and quantum cloaking. {\em New J. Phys.}
  \textbf{10}, 115024.

\bibitem{Diatta:2010:NSCM}
Diatta A, Guenneau S. 2010  Non-singular cloaks allow mimesis. {\em Journal of
  Optics} \textbf{13}, 024012.

\bibitem{Smerlak:2012:DCST}
Smerlak M. 2012  Diffusion in curved spacetimes. {\em New J. Phys.}
  \textbf{14}, 023019.

\bibitem{Farhat:2019:MCHW}
Farhat M, Guenneau S, Chen P, Alu A, Salama K. 2019  Scattering
  cancellation-based cloaking for the Maxwell-Cattaneo heat waves. {\em
  Physical Review Applied} \textbf{11}, 044089.

\end{thebibliography}

\end{document}